\documentclass[a4paper]{article}
\usepackage[scale=0.8]{geometry}
\usepackage{amsmath}
\usepackage{titlesec}
\usepackage{amsthm}
\usepackage{amssymb}
\usepackage{array}
\usepackage{cite}
\usepackage{color}
\usepackage{mathrsfs}
\usepackage{graphicx}
\usepackage{subfigure}
\usepackage[version=4]{mhchem}
\usepackage{textcomp}
\usepackage[all]{xy}
\usepackage{tikz-cd}
\usepackage{enumerate}
\usepackage{float}
\usepackage{ulem}
\usepackage[colorlinks, linkcolor=blue, citecolor=red]{hyperref}
\theoremstyle{definition}
\newtheorem{theorem}{Theorem}[section]
\newtheorem{proposition}{Proposition}[section]
\newtheorem{lemma}{Lemma}[section]
\newtheorem{corollary}{Corollary}[section]

\newtheorem{definition}{Definition}[section]
\newtheorem{example}{Example}[section]
\newtheorem{observation}{Observation}[section]
\newtheorem{main}{Main Theorem}
\theoremstyle{remark}
\newtheorem{remark}{Remark}[section]
\tikzcdset{scale cd/.style={every label/.append style={scale=#1},
    cells={nodes={scale=#1}}}}

\newcommand{\n}{\mathfrak{n}}
\newcommand{\g}{\mathfrak{g}}
\newcommand{\C}{\mathbb{C}}
\newcommand{\Hom}{\operatorname{Hom}}
\newcommand{\End}{\operatorname{End}}
\newcommand{\Z}{\mathbb{Z}}
\newcommand{\id}{\operatorname{id}}
\newcommand{\Ind}{\operatorname{Ind}}

\newcommand{\h}{\mathfrak{h}}
\newcommand{\Res}{\operatorname{Res}}
\newcommand{\p}{\mathfrak{p}}
\renewcommand{\l}{\mathfrak{l}}
\newcommand{\Ext}{\operatorname{Ext}}
\renewcommand{\O}{\mathcal{O}}
\renewcommand{\u}{\mathfrak{u}}
\newcommand{\im}{\operatorname{im}}
\newcommand{\supp}{\operatorname{supp}}
\newcommand{\s}{\mathfrak{s}}
\newcommand{\z}{\mathfrak{z}}
\newcommand{\incl}{\hookrightarrow}
\newcommand{\proj}{\twoheadrightarrow}
\renewcommand{\sl}{\mathfrak{sl}}
\newcommand{\Ann}{\operatorname{Ann}}
\newcommand{\coker}{\operatorname{coker}}
\newcommand{\I}{\mathbf{I}}
\newcommand{\R}{\mathbf{R}}

\newcommand{\J}{\mathbb{J}}

\renewcommand{\H}{\operatorname{H}}

\newcommand{\Mod}{\text{-}\mathsf{Mod}}
\newcommand{\wtMod}{\text{-}\mathsf{wtMod}}

\renewcommand{\J}{\operatorname{J}}
\renewcommand{\a}{\mathfrak{a}}
\newcommand{\fin}{\text{fin}}
\newcommand{\Oblv}{\operatorname{Oblv}}
\renewcommand{\H}{\operatorname{H}}

\begin{document}
\title{On the minimal parabolic induction}
\date{\today}
\author{Xinyu Li}

\maketitle

\begin{abstract}
    Motivated by Beilinson--Bernstein's proof of the Jantzen conjectures \cite{beilinson1993proof}, we define the minimal parabolic induction functor for Kac--Moody algebras, and establish some basic properties.
    
    As applications of the formal theory, we examine first extension groups between simple highest weight modules in the category of weight modules, and analyze the annihilators of some simple highest weight modules.
\end{abstract}

\tableofcontents

\setcounter{section}{-1}

\section{Introduction}

\subsection{Motivation}

Let $\g$ be a Kac--Moody algebra. For two weights $\lambda,\mu$ of $\g$, it is a basic problem in representation theory to calculate the first extension groups between $L(\lambda)$ and $L(\mu)$ in the category of $\g$-weight modules (or the BGG category $\O$). In particular, such problems arise in recent studies on simple affine vertex algebras (for instance, \cite{kawasetsu2022relaxed} and \cite{arakawa2023weight}).

However, the explicit result is not easy to calculate in general. Even for $\g$ being a finite dimensional semisimple Lie algebra, the answer depends on the Jantzen conjecture (cf. \cite{humphreys2008representations} Chapter 8.15.), which is a deep result in representation theory. In the finite case, the Jantzen conjecture was proved by Beilinson--Bernstein \cite{beilinson1993proof}, using their localisation theorem \cite{beilinson1981localisation}, the weight filtration of $\ell$-adic mixed perverse sheaves \cite{beilinson1982faisceaux}, and the monodromy weight filtration of nearby cycles \cite{deligne1980conjecture}.

We do not have a Beilinson--Bernstein type localisation theorem for a general Kac--Moody algebra $\g$. Therefore, some more tools need to be developed to calculate the extension groups. Instead of trying to obtain a direct formula, our idea is to reduce the calculation to some Levi subalgebra (which is usually of finite type, and hence we can apply the known results for finite dimensional reductive Lie algebras). This leads us to the construction of \textit{minimal parabolic induction}.

Let us briefly describe our main constructions and main results below.

\subsection{Main results}

Let $\g$ be a Kac--Moody algebra with a fixed choice $\Pi$ of simple roots. For any subset $\Xi\subset\Pi$, we have the corresponding standard (resp. opposite) parabolic subalgebra $\p^+_\Xi$ (resp. $\p^-_\Xi$), and the associated Levi subalgebra $\l_\Xi$.

The parabolic induction functor $$\Ind_{\Xi,!}\colon\l_\Xi\Mod\to\g\Mod,N\mapsto U\g\otimes_{U\p^+_\Xi} N,$$
which maps $\l_\Xi$-Verma modules to $\g$-Verma modules, is intensively studied in representation theory. Here we inflate an $\l_\Xi$-module $N$ to a $\p^+_\Xi$-module through the projection $\p^+_\Xi\proj\l_\Xi$. Equally important is the parabolic coinduction functor
$$\Ind_{\Xi,*}\colon\l_\Xi\Mod\to\g\Mod,N\mapsto\Hom_{U\p^-_\Xi}(U\g,N),$$
which maps completed $\l_\Xi$-coVerma modules to completed $\g$-coVerma modules. Here we inflate an $\l_\Xi$-module $N$ to a $\p^-_\Xi$-module through the projection $\p^-_\Xi\proj\l_\Xi$.

In the other direction, there is the parabolic restriction functor
$$\Res_\Xi^!\colon\g\Mod\to\l_\Xi\Mod,M\mapsto \Hom_{U\u^+_\Xi}(\C,M),$$
and the parabolic corestriction functor
$$\Res_\Xi^*\colon\g\Mod\to\l_\Xi\Mod,M\mapsto \C\otimes_{U\u^-_\Xi}M,$$
where $\u^\pm_\Xi$ is the nilradical of $\p^\pm_\Xi$. We have adjoint pairs of functors $(\Ind_{\Xi,!},\Res_\Xi^!),(\Res_\Xi^*,\Ind_{\Xi,*})$. Moreover,
$$\Res_\Xi^!\circ\Ind_{\Xi,*}=\Res_\Xi^*\circ\Ind_{\Xi,!}=\id.$$
This induces a natural transformation (see Definition~\ref{minimal-pinduction})
$$\Ind_{\Xi,!}\to\Ind_{\Xi,*}.$$

Let us denote by $\Ind_{\Xi,!*}$ the image of the above transformation. We call $\Ind_{\Xi,!*}$ the \textit{minimal parabolic induction}, or the \textit{intermediate parabolic induction}.

In this paper, we establish some basic properties of minimal parabolic induction. For example, like the construction of IC sheaves, one may expect that $\Ind_{\Xi,!*}$ sends ``good'' (lisse) simple objects to simple objects. We confirm this expectation in Proposition~\ref{ind-simple} by showing that $\Ind_{\Xi,!*}$ maps a simple weight module to a simple weight module (see Section~\ref{weight-module} for the definition of weight modules and the category $\g\wtMod$).

\begin{main}
    For any simple $\l_\Xi$-weight module $N$, $\Ind_{\Xi,!*}(N)$ is a simple $\g$-weight module.
\end{main}

We exhibit two applications of the general theory in Section~\ref{applications}. First, we explore the behavior of first extension groups between some simple highest weight modules under minimal parabolic induction. Under some rather technical conditions on two weights $\lambda,\mu$, we are able to prove that there is an isomorphism (Proposition~\ref{ind-ext})
$$\Ext^1_{\g\wtMod}(L(\lambda),L(\mu))=\Ext^1_{\l_\Xi\wtMod}(L_\Xi(\lambda),L_\Xi(\mu)).$$

\begin{main}
    Let $\mu,\lambda$ be two weights of $\g$ such that $\mu-\lambda\in\Z\Xi$, then we always have an inclusion
    $$\Ext^1_{\g\wtMod}(L(\lambda),L(\mu))\incl\Ext^1_{\l_\Xi\wtMod}(L_\Xi(\lambda),L_\Xi(\mu)).$$
    Moreover, suppose that $\mu-\lambda\notin\Z_{\geq0}\Xi$ and $\lambda$ is $\Xi$-joyful (a technical notion introduced in Definition~\ref{joyful}), then the above inclusion is an isomorphism.
\end{main}

Another application is about annihilators. Assume $\l_\Xi$ is a finite dimensional reductive Lie algebra, then we are able to prove that
$$\Ann_{U\g}(\Ind_{\Xi,!*}(M_\Xi(\lambda)))=\Ann_{U\g}(\Ind_{\Xi,!*}(M_\Xi(w\cdot\lambda)))$$
for $\lambda$ being a $\rho$-anti-dominant integral weight for $\l_\Xi$ and $w\in W_\Xi$, the Weyl group of $\l_\Xi$. As a corollary, we show that under the same assumption, $\Ann_{U\g}(L(\lambda))\subset\Ann_{U\g}(L(w\cdot\lambda))$.

\begin{main}
    Suppose $\Xi$ is of finite type, i.e. $\l_\Xi$ is a finite dimensional reductive Lie algebra. Let $\lambda\in\h^*$ be a weight such that $\langle\lambda+\rho,\alpha^\vee\rangle\in\Z_{\leq0}$ for any $\alpha^\vee\in\Xi^\vee$, then we have
    $$\Ann_{U\g}(L(\lambda))=\Ann_{U\g}(\Ind_{\Xi,!*}(M_\Xi(w\cdot\lambda)))\subset\Ann_{U\g}(L(w\cdot\lambda))$$
    for any $w\in W_\Xi$, the Weyl group of $\l_\Xi$.
\end{main}

To the knowledge of the author, this result is new for arbitrary Kac--Moody algebra.\footnote{When $\g$ is a finite dimensional semisimple Lie algebra, this result was proved by Vogan \cite{vogan1980ordering}. When $\g$ is an affine Kac--Moody algebra, this result was claimed by Dhillon, based on his work \cite{campbell2021affine}. Both of them used Harish-Chandra bimodules, whose general theory has not been developed for arbitrary Kac--Moody algebras yet.}

\subsection{Convention}

Throughout the paper, we work over an algebraically closed field $\C$ of characteristic $0$. Symbols $\otimes$ and $\Hom$ without subscripts mean the corresponding operations in the category of $\C$-vector spaces, that is, $\otimes_\C$ and $\Hom_\C$.

For a ring (or a Lie algebra) $\Lambda$, we use $\Lambda\Mod$ to denote the category of left $\Lambda$-modules. By a module, we mean a left module.

\subsection{Acknowledgement}

The author would like to express his deepest gratitude to his advisors, Prof. Wenbin Yan and Prof. Peng Shan, without whose help the current work would have never been done. He especially thanks Peng Shan for pointing out many inaccuracies, due to the author's arrogance and over-optimism, in an earlier draft of the paper.

There is a similar construction in the theory of vertex algebras, namely the Zhu's induction. The author thanks Tomoyuki Arakawa, Thomas Creutzig, and Yongchang Zhu for correspondences on this subject.

The author thanks Gurbir Dhillon for teaching him parabolic induction in categorical representation theory, Dingxin Zhang for teaching him $\ell$-adic mixed perverse sheaves, and Qixian Zhao for teaching him Duflo's theorem.

\section{Kac--Moody algebra setup}
\subsection{Parabolic type}

Let $\g$ be a Kac--Moody algebra associated with the triple $(\h,\Pi,\Pi^\vee)$, where $\h$ is a fixed Cartan subalgebra, $\Pi\subset\h^*$ (resp. $\Pi^\vee\subset\h$) is the collection of simple roots (resp. coroots). Denoted by $\Delta$ the set of roots of $\g$, $\Delta^+$ (resp. $\Delta^-$) the set of positive (resp. negative) roots of $\g$, we have the root decomposition
$$\g=\h\oplus\bigoplus_{\alpha\in\Delta}\g_\alpha=\n^-\oplus\h\oplus\n^+,$$
where
$$\n^+=\bigoplus_{\alpha\in\Delta^+}\g_\alpha,\n^-=\bigoplus_{\alpha\in\Delta^-}\g_\alpha.$$

Any subset $\Xi$ of $\Pi$ (together with the corresponding subset $\Xi^\vee$ of $\Pi^\vee$) defines a \textit{parabolic type}. More precisely, let $\Delta_\Xi=\Delta\cap\Z\Xi,\Delta^\pm_\Xi=\Delta_\Xi\cap\Delta^\pm$. Then we have the standard (opposite) parabolic subalgebra of type $\Xi$
$$\p^+_\Xi=\bigoplus_{\alpha\in\Delta^-_\Xi}\g_\alpha\oplus\h\oplus\n^+,\p^-_\Xi=\bigoplus_{\alpha\in\Delta^+_\Xi}\g_\alpha\oplus\h\oplus\n^-,$$
the (opposite) nilradical
$$\u^\pm_\Xi=\bigoplus_{\alpha\in\Delta^\pm\backslash\Delta^\pm_\Xi}\g_\alpha,$$
and the Levi
$$\l_\Xi=\h\oplus\bigoplus_{\alpha\in\Delta_\Xi}\g_\alpha.$$

\subsection{Transpose anti-involution}
\label{transpose-anti-involution}

Let us name the simple roots of $\g$ by $\alpha_1,\cdots,\alpha_n$. We fix the Chevalley generators $e_i\in\g_{\alpha_i}$, $f_i\in\g_{-\alpha_i}$. From the structure theory of Kac--Moody algebra (the Serre relations), it is well-known that the map $\tau(e_i)=f_i,\tau(f_i)=e_i$ and $\tau(h)=h$ for $h\in\h$ extends to an anti-involution of $U\g$, which we still denote by $\tau$. This anti-involution $\tau$ is called the \textit{transpose anti-involution}, in the sense that for $\g=\sl_n$ with the standard choice of Chevalley generators, $\tau(x)=x^t$ is the transpose of a matrix.

The transpose anti-involution interchanges $\n^+$ and $\n^-$. More generally, for any parabolic type $\Xi$, $\tau$ interchanges $\p^+_\Xi$ and $\p^-_\Xi$, $\u^+_\Xi$ and $\u^-_\Xi$, and restricts to an anti-involution on $\l_\Xi$.

The transpose anti-involution $\tau$ leads to an isomorphism of algebra $U\g\simeq(U\g)^{\text{op}}$. As a consequence, we can identify the category of left $U\g$-modules with the category of right $U\g$-modules via the twisting of $\tau$. More precisely, for any left $U\g$-module $M$, we can define a right $U\g$-module structure on $M$ by
$$m.X=\tau(X).m,X\in U\g,m\in M.$$
Conversely, for any right $U\g$-module $M$, we can define a left $U\g$-module structure on $M$ by
$$X.m=m.\tau(X),X\in U\g,m\in M.$$

\subsection{Basic representations}

Let $\lambda\in\h^*$ be a weight. We have three basic types of $\g$-modules, that is, the Verma module $M(\lambda)$ of highest weight $\lambda$, the simple module $L(\lambda)$ of highest weight $\lambda$, and the completed coVerma module $M(\lambda)^*$\footnote{The usual coVerma module, or dual Verma module, is $M(\lambda)^\vee$. Here $^\vee$ is the restricted dual introduced in Section~\ref{km-weight}.}. It is known that $M(\lambda)$ has a maximal proper $\g$-submodule, usually denoted by $N(\lambda)$. There is a canonical (up to a nonzero constant) morphism $M(\lambda)\to M(\lambda)^*$ (that is induced by the Shapovalov form on $M(\lambda)$), whose kernel is $N(\lambda)$ and whose image is $M(\lambda)/N(\lambda)\simeq L(\lambda)$.

Similarly, let us denote by $M_\Xi(\lambda),L_\Xi(\lambda)$ and $M_\Xi(\lambda)^*$, respectively, the $\l_\Xi$-Verma module of highest weight $\lambda$, the $\l_\Xi$-simple module of highest weight $\lambda$, and the $\l_\Xi$-completed coVerma module. Let $N_\Xi(\lambda)\subset M_\Xi(\lambda)$ be the maximal proper $\l_\Xi$-submodule of $M_\Xi(\lambda)$, then $M_\Xi(\lambda)/N_\Xi(\lambda)\simeq L_\Xi(\lambda)$.

\subsection{Geometry of weight}

Let $W$ be the Weyl group of $\g$. We fix an element $\rho\in\h^*$ such that $\langle\rho,\alpha^\vee\rangle=1$ for any $\alpha^\vee\in\Pi^\vee$. The \textit{dot action} of $W$ on $\h^*$ is defined by
$$w\cdot\lambda=w(\lambda+\rho)-\rho.$$

For any root $\alpha$ of $\g$, let us denote by $s_\alpha$ the corresponding reflection in $W$. Let $W_\Xi$ be the subgroup of $W$ generated by $s_\alpha$ for $\alpha\in\Xi$. It is the Weyl group of $\l_\Xi$.

For two weights $\mu,\lambda\in\h^*$, we write $\mu\geq\lambda$ if $\mu-\lambda\in\Z_{\geq0}\Pi$.

Let $Q_\Xi$ be the root lattice of $\l_\Xi$. The following observation will be useful when dealing with weight modules in Section~\ref{weight-module}.

\begin{observation}
    \label{weight}
    The root lattice $Q_\Xi$ of $\l_\Xi$ has zero intersection with the monoid spanned by $\Delta^-\backslash\Delta^-_\Xi$.
\end{observation}

\section{Minimal parabolic induction}
\label{mpind}

Let us fix a parabolic type $\Xi\subset\Pi$.

\begin{definition}[parabolic restriction]
    There are two types of parabolic restriction functors. The \textit{parabolic $!$-restriction} (invariant) functor is defined by
    $$\Res_\Xi^!\colon\g\Mod\to\l_\Xi\Mod,M\mapsto\Hom_{U\u^+_\Xi}(\C,M),$$
    while the \textit{parabolic $*$-restriction} (coinvariant) functor is defined by
    $$\Res_\Xi^*\colon\g\Mod\to\l_\Xi\Mod,M\mapsto\C\otimes_{U\u^-_\Xi}M.$$
    Here $\C$ means the trivial representation. For a $\g$-module $M$, the $\l_\Xi$-action on $\Res_\Xi^?(M)$ ($?\in\{!,*\}$) is inherited from the $\l_\Xi$-action on $M$. It is well-defined because $[\l_\Xi,\u^\pm_\Xi]\subset\u^\pm_\Xi$.
\end{definition}

\begin{definition}
    There are two types of parabolic induction functors.
    
    The \textit{parabolic $!$-induction} functor is defined by
    $$\Ind_{\Xi,!}\colon\l_\Xi\Mod\to\g\Mod,N\mapsto U\g\otimes_{U\p^+_\Xi}N.$$
    Here we inflate an $\l_\Xi$-module $N$ to a $\p^+_\Xi$-module via the projection $\p^+_\Xi\proj\l_\Xi$.
    
    The \textit{parabolic $*$-induction} (coinduction) functor is defined by
    $$\Ind_{\Xi,*}\colon\l_\Xi\Mod\to\g\Mod,N\mapsto\Hom_{U\p^-_\Xi}(U\g,N).$$
    Here we inflate an $\l_\Xi$-module $N$ to a $\p^-_\Xi$-module via the projection $\p^-_\Xi\proj\l_\Xi$. The left multiplication of $U\g$ and the right multiplication of $U\p^+_\Xi$ make $U\g$ a $(U\g,U\p^+_\Xi)$-bimodule. Then we view $U\g$ as a $(U\p^-_\Xi,U\g)$-bimodule via the transpose anti-involution.
\end{definition}

Here we collect some well-known facts about parabolic restrictions and parabolic inductions.

\begin{proposition}
    \begin{enumerate}
        \item We have adjoint pairs of functors $(\Ind_{\Xi,!},\Res_\Xi^!),(\Res_\Xi^*,\Ind_{\Xi,*})$.
        \item $\Res_\Xi^!\circ\Ind_{\Xi,*}=\Res_\Xi^*\circ\Ind_{\Xi,!}=\id$.
        \item The functor $\Ind_{\Xi,?}$ ($?\in\{!,*\}$) is exact, the functor $\Res_\Xi^!$ (resp. $\Res_\Xi^*$) is left (resp. right) exact.
        \item The $!$-induction $\Ind_{\Xi,!}$ maps an $\l_\Xi$-Verma module to the $\g$-Verma module with the same highest weight, while the $*$-induction $\Ind_{\Xi,*}$ maps a completed $\l_\Xi$-coVerma module to the completed $\g$-coVerma module associated to the same weight.
    \end{enumerate}
\end{proposition}

\begin{proof}
    By the $\otimes$-$\Hom$ adjunction, for an $\l_\Xi$-module $N$ and a $\g$-module $M$, we have canonical isomorphisms
    $$\Hom_{U\g}(U\g\otimes_{U\p^+_\Xi}N,M)=\Hom_{U\p^+_\Xi}(N,M)=\Hom_{U\l_\Xi}(N,\Hom_{U\u^+_\Xi}(\C,M)).$$
    This shows that $(\Ind_{\Xi,!},\Res_\Xi^!)$ is an adjoint pair. Similarly,
    $$\Hom_{U\l_\Xi}(\C\otimes_{U\u^-_\Xi}M,N)=\Hom_{U\p^-_\Xi}(M,N)=\Hom_{U\g}(M,\Hom_{U\p^-_\Xi}(U\g,N)).$$
    This shows that $(\Res^*_\Xi,\Ind_{\Xi,*})$ is an adjoint pair.

    Moreover, for any $\l_\Xi$-module $N$,
    $$\Hom_{U\u^+_\Xi}(\C,\Hom_{U\p^-_\Xi}(U\g,N))=\Hom_{U\p^-_\Xi}(U\g\otimes_{U\u^+_\Xi}\C,N)=\Hom_{U\p^-_\Xi}(U\p^-_\Xi,N)=N.$$
    This shows that $\Res^!_\Xi\circ\Ind_{\Xi,*}=\id$.

    Similarly,
    $$\C\otimes_{U\u^-_\Xi}U\g\otimes_{U\p^+_\Xi}N=U\p^+_\Xi\otimes_{U\p^+_\Xi}N=N.$$
    This shows that $\Res^*_\Xi\circ\Ind_{\Xi,!}=\id$.

    By the PBW theorem, $U\g=U\u^-_\Xi\otimes U\p^+_\Xi$ is a free (hence projective, flat) right $U\p^+_\Xi$-module. This shows the exactness of $\Ind_{\Xi,?}$. As a right (resp. left) adjoint, $\Res_\Xi^!$ (resp. $\Res_\Xi^*$) is left (resp. right) exact.

    Recall that Verma modules (resp. completed coVerma modules) are constructed by $!$-induction (resp. $*$-induction) with respect to the empty collection $\emptyset\subset\Pi$. The last assertion follows from the composability of $!$-induction (resp. $*$-induction), which is again due to the composability of $\otimes$ (resp. $\Hom$).
\end{proof}

Now we see that there is a natural transformation from $\Ind_{\Xi,!}$ to $\Ind_{\Xi,*}$ fitting into the following diagram
$$\begin{tikzcd}
{\Ind_{\Xi,!}} \arrow[Rightarrow, r, no head] \arrow[d]                       & {\Ind_{\Xi,!}\circ\Res^!_\Xi\circ\Ind_{\Xi,*}} \arrow[d] \\
{\Ind_{\Xi,*}\circ\Res^*_\Xi\circ\Ind_{\Xi,!}} \arrow[Rightarrow, r, no head] & {\Ind_{\Xi,*}}                                          
\end{tikzcd}$$

\begin{definition}
    \label{minimal-pinduction}
    Let $\Ind_{\Xi,!*}\colon\l_\Xi\Mod\to\g\Mod$ be the image of the natural transformation $\Ind_{\Xi,!}\to\Ind_{\Xi,*}$, that is (on the object level),
    $$\Ind_{\Xi,!*}(N)=\im(\Ind_{\Xi,!}(N)\to\Ind_{\Xi,*}(N)).$$
    We call $\Ind_{\Xi,!*}$ the \textit{minimal parabolic induction} functor, or the \textit{intermediate parabolic induction} functor.
\end{definition}

\begin{proposition}
    \label{inj-surj}
    Let $N_1\proj N_2$ be a surjection in $\l_\Xi\Mod$, then the induced map $\Ind_{\Xi,!*}(N_1)\to\Ind_{\Xi,!*}(N_2)$ is also surjective.

    Dually, let $N_1'\incl N_2'$ be an injection in $\l_\Xi\Mod$, then the induced map $\Ind_{\Xi,!*}(N_1')\to\Ind_{\Xi,!*}(N_2')$ is also injective.
\end{proposition}

\begin{proof}
    Since $\Ind_{\Xi,!}$ is exact, the map $\Ind_{\Xi,!}(N_1)\to\Ind_{\Xi,!}(N_2)$ is surjective. Consider the commutative diagram
    $$\begin{tikzcd}
{\Ind_{\Xi,!}(N_1)} \arrow[d, two heads] \arrow[r, two heads] & {\Ind_{\Xi,!}(N_2)} \arrow[d, two heads] \\
{\Ind_{\Xi,!*}(N_1)} \arrow[r]                                & {\Ind_{\Xi,!*}(N_2)}                    
\end{tikzcd}$$
    The surjectivity of the bottom arrow follows.

    Dually, the map $\Ind_{\Xi,*}(N_1')\to\Ind_{\Xi,*}(N_2')$ is injective because $\Ind_{\Xi,*}$ is exact. Consider the commutative diagram
    $$\begin{tikzcd}
{\Ind_{\Xi,!*}(N_1')} \arrow[r] \arrow[d, hook] & {\Ind_{\Xi,!*}(N_2')} \arrow[d, hook] \\
{\Ind_{\Xi,*}(N_1')} \arrow[r, hook]                        & {\Ind_{\Xi,*}(N_2')}                 
\end{tikzcd}$$
    The injectivity of the top arrow follows.
\end{proof}

\begin{definition}
    Let $\J_{\Xi,!}^{-1}\colon\l_\Xi\Mod\to\g\Mod$ be the kernel of the natural transformation $\Ind_{\Xi,!}\proj\Ind_{\Xi,!*}$, $\J_{\Xi,*}^1\colon\l_\Xi\Mod\to\g\Mod$ be the cokernel of the natural transformation $\Ind_{\Xi,!*}\incl\Ind_{\Xi,*}$.\footnote{$\J$ stands for Jantzen.}
\end{definition}

Let $N$ be an $\l_\Xi$-module, then the map $\Ind_{\Xi,!}(N)\to\Ind_{\Xi,*}(N)$ can be explicitly described as follows. Using the PBW decomposition $U\g=U\u^-_\Xi\otimes U\p^+_\Xi$, we can identify $U\g\otimes_{U\p^+_\Xi}N$ with $U\u^-_\Xi\otimes N$, and $\Hom_{U\p^-_\Xi}(U\g,N)$ with $\Hom(U\u^-_\Xi,N)$. Let
$$\epsilon^\pm\colon U\u^\pm_\Xi\proj U\u^\pm_\Xi/\u^\pm_\Xi(U\u^\pm_\Xi)=\C$$
be the augmentation maps, $\phi=\epsilon^-\otimes\id\otimes\epsilon^+$ be the map
$$\phi=\epsilon^-\otimes\id\otimes\epsilon^+\colon U\g=U\u^-_\Xi\otimes U\l_\Xi\otimes U\u^+_\Xi\to\C\otimes U{\l_\Xi}\otimes\C =U\l_\Xi.$$
Then the map $\Ind_{\Xi,!}(N)\to\Ind_{\Xi,*}(N)$ can be identified with
$$U\u^-_\Xi\otimes N\to\Hom(U\u^-_\Xi,N),X\otimes n\mapsto[Y\mapsto\phi(\tau(X)Y)n=\phi(\tau(Y)X)n].$$

\begin{proposition}
    \label{j-1}
    Let $N$ be an $\l_\Xi$-module. Under the identification
    $$\Ind_{\Xi,!}(N)=U\g\otimes_{U\p^+_\Xi}N=U\u^-_\Xi\otimes N,$$
    we have $\J^{-1}_{\Xi,!}(N)\subset\u^-_\Xi(U\u^-_\Xi)\otimes N$.

    Under the identification
    $$\Ind_{\Xi,*}(N)=\Hom_{U\p^-_\Xi}(U\g,N)=\Hom(U\u^-_\Xi,N),$$
    we have $\Hom(\u^-_\Xi(U\u^-_\Xi),N)\proj\J^1_{\Xi,*}(N)$.
\end{proposition}

\begin{proof}
    Let us consider a commutative diagram with exact rows
    $$\begin{tikzcd}
0 \arrow[r] & {\J^{-1}_{\Xi,!}(N)} \arrow[r] \arrow[dd, dashed] & {\Ind_{\Xi,!}(N)} \arrow[r] \arrow[Rightarrow, dd, no head] & {\Ind_{\Xi,!*}(N)} \arrow[r] \arrow[d] & 0 \\
            &                                                   &                                                             & {\Ind_{\Xi,*}(N)} \arrow[d]            &   \\
0 \arrow[r] & \u^-_\Xi(U\u^-_\Xi)\otimes N \arrow[r]            & U\u^-_\Xi\otimes N \arrow[r]                                & \C\otimes N=N \arrow[r]                & 0
\end{tikzcd}$$
    The existence of the dashed arrow follows from the exactness of rows and the commutativity of the diagram.

    For the dual statement, notice that we have a commutative diagram with exact rows
    $$\begin{tikzcd}
0 \arrow[r] & {\Hom(\C,N)=N} \arrow[d] \arrow[r] & {\Hom(U\u^-_\Xi,N)} \arrow[r] \arrow[Rightarrow, dd, no head] & {\Hom(\u^-_\Xi(U\u^-_\Xi),N)} \arrow[r] \arrow[dd, dashed] & 0 \\
            & {\Ind_{\Xi,!}(N)} \arrow[d]        &                                                               &                                                           &   \\
0 \arrow[r] & {\Ind_{\Xi,!*}(N)} \arrow[r]       & {\Ind_{\Xi,*}(N)} \arrow[r]                                   & {\J^1_{\Xi,*}(N)} \arrow[r]                               & 0
\end{tikzcd}$$
The existence of the dashed arrow follows from the exactness of rows and the commutativity of the diagram.
\end{proof}

\begin{proposition}
    \label{j-2}
    Let $N$ be an $\l_\Xi$-module. For any $\g$-submodule $M\subset\Ind_{\Xi,!}(N)$ that is contained in $\u^-_\Xi(U\u^-_\Xi)\otimes N$ under the identification $\Ind_{\Xi,!}(N)=U\u^-_\Xi\otimes N$, we have $M\subset\J^{-1}_{\Xi,1}(N)$. Dually, any quotient $\g$-module $\Ind_{\Xi,*}(N)\proj M'$ factorizing through $\Hom(\u^-_\Xi(U\u^-_\Xi),N)$ under the identification $\Ind_{\Xi,*}(N)=\Hom(U\u^-_\Xi,N)$ is a quotient of $\J^1_{\Xi,*}(N)$.
\end{proposition}

\begin{proof}
    Every element in $M$ can be written as a finite sum $\sum_i X_i\otimes n_i$ for some $X_i\in U\u^-_\Xi$ and $n_i\in N$. To show that it lies in
    $$\J^{-1}_{\Xi,!}(N)=\ker(\Ind_{\Xi,!}(N)\to\Ind_{\Xi,*}(N)),$$
    it suffices to show that $\sum_i\phi(\tau(Y)X_i)n_i=0$ for any $Y\in U\g$. Let us decompose each $\tau(Y)X_i$ as a finite sum $\tau(Y)X_i=\sum_j X_{i,j}^-X_{i,j}^0X_{i,j}^+$ for some $X_{i,j}^-\in U\u^-_\Xi,X_{i,j}^0\in U\l_\Xi,X_{i,j}^+\in U\u^+_\Xi$, then
    $$\sum_i\phi(\tau(Y)X_i)n_i=\sum_{i,j}\phi(X_{i,j}^-X_{i,j}^0X_{i,j}^+)n_i=\sum_{i,j}\epsilon^-(X_{i,j}^-)X_{i,j}^0\epsilon^+(X_{i,j}^+)n_i.$$
    Notice that the $\tau(Y)$ action on $\sum_i X_i\otimes n_i$ is
    $$\sum_i\tau(Y)X_i\otimes_{U\p^+_\Xi}n_i=\sum_{i,j}X_{i,j}^-X_{i,j}^0X_{i,j}^+\otimes_{U\p^+_\Xi}n_i=\sum_{i,j}X_{i,j}^-\otimes X_{i,j}^0\epsilon^+(X_{i,j}^+)n_i.$$
    By assumption, it is contained in $M\subset\u^-_\Xi(U\u^-_\Xi)\otimes N$. This means exactly
    $$\sum_{i,j}\epsilon^-(X_{i,j}^-)X_{i,j}^0\epsilon^+(X_{i,j}^+)n_i=0.$$
    Therefore, we conclude that $\sum_i\phi(\tau(Y)X_i)n_i=0$.

    Dually, to show that $M'$ is a quotient of
    $$\J^1_{\Xi,*}(N)=\coker(\Ind_{\Xi,!}(N)\to\Ind_{\Xi,*}(N)),$$
    it suffices to show that the composition
    $$\Ind_{\Xi,!}(N)\to\Ind_{\Xi,*}(N)\to M'$$
    vanishes. As a $U\g$-module, $\Ind_{\Xi,!}(N)$ is generated by $1\otimes n\in 1\otimes N=N$. Hence it suffices to show that the image of $1\otimes n$ vanishes in $M'$. Notice that we have a commutative diagram
    $$\begin{tikzcd}
{\Ind_{\Xi,!}(N)} \arrow[r] \arrow[Rightarrow, d, no head] & {\Ind_{\Xi,*}(N)} \arrow[r, two heads]   & M'                                                 \\
U\u^-_\Xi\otimes N \arrow[r]                               & {\Hom(U\u^-_\Xi,N)} \arrow[r, two heads] & {\Hom(\u^-_\Xi(U\u^-_\Xi),N)} \arrow[u, two heads]
\end{tikzcd}$$
    The image of $1\otimes n$ in $\Hom(\u^-_\Xi(U\u^-_\Xi),N)$ vanishes, so its image in $M'$ also vanishes. We are done.
\end{proof}

\begin{remark}
    \label{max}
    Combining Proposition~\ref{j-1}~and~\ref{j-2}, we see that for any $\l_\Xi$-module $N$, $\J^{-1}_{\Xi,!}(N)$ is \textit{the} maximal $\g$-submodule of $\Ind_{\Xi,!}(N)=U\u^-_\Xi\otimes N$ that is contained in $\u^-_\Xi(U\u^-_\Xi)\otimes N$, and $\J^1_{\Xi,*}(N)$ is \textit{the} maximal $\g$-module quotient of $\Ind_{\Xi,*}(N)=\Hom(U\u^-_\Xi,N)$ that factors through $\Hom(\u^-_\Xi(U\u^-_\Xi),N)$.
\end{remark}

Let $\operatorname{Oblv}^\g_{\l_\Xi}$ be the forgetful functor from $\g\Mod$ to $\l_\Xi\Mod$. There are canonical natural transformations
$$\Res_\Xi^!\incl\Oblv^\g_{\l_\Xi},\Oblv^\g_{\l_\Xi}\proj\Res_\Xi^*,$$
the composition of which induces a natural transformation $\Res_\Xi^!\to\Res_\Xi^*$.

\begin{definition}
    Let $\Res_\Xi^{!*}\colon\g\Mod\to\l_\Xi\Mod$ be the image of the natural transformation $\Res_\Xi^!\to\Res_\Xi^*$. It seems plausible to call $\Res_\Xi^{!*}$ the \textit{intermediate parabolic restriction} functor.
\end{definition}

Now we have a commutative diagram of natural transformations
$$\begin{tikzcd}
\id \arrow[r] \arrow[Rightarrow, rrrr, no head, bend left]  \arrow[Rightarrow, dd, no head]                     & {\Res_\Xi^!\circ\Ind_{\Xi,!}} \arrow[r] \arrow[d, two heads] & {\Res_\Xi^!\circ\Ind_{\Xi,!*}} \arrow[r, hook] \arrow[d, two heads] & {\Res_\Xi^!\circ\Ind_{\Xi,*}} \arrow[Rightarrow, r, no head] \arrow[d, two heads] & \id \arrow[Rightarrow, dd, no head] \\
     & {\Res^{!*}_\Xi\circ\Ind_{\Xi,!}} \arrow[r] \arrow[d, hook]   & {\Res^{!*}_\Xi\circ\Ind_{\Xi,!*}} \arrow[r] \arrow[d, hook]         & {\Res^{!*}_\Xi\circ\Ind_{\Xi,*}} \arrow[d, hook]                                  &     \\
\id \arrow[Rightarrow, r, no head] \arrow[Rightarrow, rrrr, no head, bend right] & {\Res^*_\Xi\circ\Ind_{\Xi,!}} \arrow[r, two heads]           & {\Res^*_\Xi\circ\Ind_{\Xi,!*}} \arrow[r]                            & {\Res^*_\Xi\circ\Ind_{\Xi,*}} \arrow[r]                                           & \id
\end{tikzcd}$$
Here the natural transformations $\id\to\Res^!_\Xi\circ\Ind_{\Xi,!}$ and $\Res^*_\Xi\circ\Ind_{\Xi,*}\to\id$ are induced by adjunctions.

The map $\Res^!_\Xi\circ\Ind_{\Xi,!*}\to\Res^!_\Xi\circ\Ind_{\Xi,*}$ (resp. $\Res^*_\Xi\circ\Ind_{\Xi,!}\to\Res^*_\Xi\circ\Ind_{\Xi,!*}$) is monic (resp. epic) because $\Res^!_\Xi$ (resp. $\Res^*_\Xi$) is left (resp. right) exact.

\begin{proposition}
    \label{res-ind}
    We have $\Res^!_\Xi\circ\Ind_{\Xi,!*}=\Res^*_\Xi\circ\Ind_{\Xi,!*}=\Res^{!*}_\Xi\circ\Ind_{\Xi,!*}=\id$.
\end{proposition}

\begin{proof}
    Consider the following commutative diagram of natural transformations
    $$\begin{tikzcd}
            & \id \arrow[r]                                     & {\Res^!_\Xi\circ\Ind_{\Xi,!}} \arrow[r] \arrow[d, hook] & {\Res^!_\Xi\circ\Ind_{\Xi,!*}} \arrow[d, hook]  &   \\
0 \arrow[r] & {\Oblv^\g_{\l_\Xi}\circ\J^{-1}_{\Xi,!}} \arrow[r] & {\Oblv^\g_{\l_\Xi}\circ\Ind_{\Xi,!}} \arrow[r]          & {\Oblv^\g_{\l_\Xi}\circ\Ind_{\Xi,!*}} \arrow[r] & 0
\end{tikzcd}$$
    whose bottom row is exact. We claim that the composition
    $$\id\to\Res^!_\Xi\circ\Ind_{\Xi,!}\to\Res^!_\Xi\circ\Ind_{\Xi,!*}\to\Oblv^\g_{\l_\Xi}\circ\Ind_{\Xi,!*}$$
    is monic. In fact, for any $\l_\Xi$-module,
    $$\im(N\to\Res^!_\Xi\circ\Ind_{\Xi,!}(N)\to\Oblv^\g_{\l_\Xi}\circ\Ind_{\Xi,!}(N))=1\otimes N$$
    under the identification $\Oblv^\g_{\l_\Xi}\circ\Ind_{\Xi,!}(N)=U\u^-_\Xi\otimes N$. By Proposition~\ref{j-1}, we see that it has zero intersection with
    $$\Oblv^\g_{\l_\Xi}\circ\J^{-1}_{\Xi,!}(N)\subset\u^-_\Xi(U\u^-_\Xi)\otimes N.$$
    This shows that the composition
    $$\id\to\Res^!_\Xi\circ\Ind_{\Xi,!}\to\Oblv^\g_{\l_\Xi}\circ\Ind_{\Xi,!}\to\Oblv^\g_{\l_\Xi}\circ\Ind_{\Xi,!*}$$
    is monic, i.e. the composition
    $$\id\to\Res^!_\Xi\circ\Ind_{\Xi,!}\to\Res^!_\Xi\circ\Ind_{\Xi,!*}\to\Oblv^\g_{\l_\Xi}\circ\Ind_{\Xi,!*}$$
    is monic. Consequently, the composition
    $$\id\to\Res^!_\Xi\circ\Ind_{\Xi,!}\to\Res^!_\Xi\circ\Ind_{\Xi,!*}$$
    is monic. Now we have two monos
    $$\id\incl\Res^!_\Xi\circ\Ind_{\Xi,!*},\Res^!_\Xi\circ\Ind_{\Xi,!*}\to\Res^!_\Xi\circ\Ind_{\Xi,*}=\id$$
    whose composition is the identity. This shows that they are isomorphisms, and $\Res^!_\Xi\circ\Ind_{\Xi,!*}=\id$.

    Dually, consider the following commutative diagram of natural transformations
    $$\begin{tikzcd}
0 \arrow[r] & {\Oblv^\g_{\l_\Xi}\circ\Ind_{\Xi,!*}} \arrow[r] \arrow[d, two heads] & {\Oblv^\g_{\l_\Xi}\circ\Ind_{\Xi,*}} \arrow[r] \arrow[d, two heads] & {\Oblv^\g_{\l_\Xi}\circ\J^1_{\Xi,*}} \arrow[r] & 0 \\
            & {\Res^*_\Xi\circ\Ind_{\Xi,!*}} \arrow[r]                             & {\Res^*_\Xi\circ\Ind_{\Xi,*}} \arrow[r]                             & \id                                            &  
\end{tikzcd}$$
    whose top row is exact. We claim that the composition
    $$\Res^*_\Xi\circ\Ind_{\Xi,!*}\to\Res^*_\Xi\circ\Ind_{\Xi,*}\to \id$$
    is epic. In fact, for any $\l_\Xi$-module $N$, the composition
    $$\Oblv^\g_{\l_\Xi}\circ\Ind_{\Xi,*}(N)\to\Res^*_\Xi\circ\Ind_{\Xi,*}(N)\to N$$
    is identified with the quotient
    $$\Hom(U\u^-_\Xi,N)\proj\Hom(\C,N),$$
    whose kernel is identified with $\Hom(\u^-_\Xi(U\u^-_\Xi),N)$. By Proposition~\ref{j-1}, $\Hom(\u^-_\Xi(U\u^-_\Xi),N)$ maps surjectively to $\Oblv^\g_{\l_\Xi}\circ\J^1_{\Xi,*}(N)$. This shows that the composition
    $$\Oblv^\g_{\l_\Xi}\circ\Ind_{\Xi,!*}\to\Oblv^\g_{\l_\Xi}\circ\Ind_{\Xi,*}\to\Res^*_\Xi\circ\Ind_{\Xi,*}\to\id$$
    is epic, i.e. the composition
    $$\Oblv^\g_{\l_\Xi}\circ\Ind_{\Xi,!*}\to\Res^*_\Xi\circ\Ind_{\Xi,!*}\to\Res^*_\Xi\circ\Ind_{\Xi,*}\to\id$$
    is epic. Consequently, the composition
    $$\Ind_{\Xi,!*}\to\Res^*_\Xi\circ\Ind_{\Xi,!*}\to\Res^*_\Xi\circ\Ind_{\Xi,*}\to\id$$
    is epic. Now we have two epis
    $$\id=\Res^*_\Xi\circ\Ind_{\Xi,!}\proj\Res^*_\Xi\circ\Ind_{\Xi,!*},\Res^*_\Xi\circ\Ind_{\Xi,!*}\proj\id$$
    whose composition is the identity. This shows that they are isomorphisms, and $\Res^*_\Xi\circ\Ind_{\Xi,!*}=\id$.

    Now $$\Res^{!*}_\Xi\circ\Ind_{\Xi,!*}=\im(\Res^!_\Xi\circ\Ind_{\Xi,!*}\to\Res^*_\Xi\circ\Ind_{\Xi,!*})=\id,$$
    we are done.
\end{proof}

\begin{remark}
\label{exactness-minimal}
    As we have seen, the functors $\Ind_{\Xi,!}$ and $\Ind_{\Xi,*}$ are exact. However, $\Ind_{\Xi,!*}$ is not exact in general. Let $N^\bullet$ be an acyclic cochain complex of $\l_\Xi$-modules, then we have a short exact sequence of cochain complexes
$$0\to\J^{-1}_{\Xi,!}(N^\bullet)\to\Ind_{\Xi,!}(N^\bullet)\to\Ind_{\Xi,!*}(N^\bullet)\to 0.$$
This induces a long exact sequence of cohomologies
$$\cdots\to\H^i(\Ind_{\Xi,!}(N^\bullet))\to\H^i(\Ind_{\Xi,!*}(N^\bullet))\to\H^{i+1}(\J^{-1}_{\Xi,!}(N^\bullet))\to\H^{i+1}(\Ind_{\Xi,!}(N^\bullet))\to\cdots.$$

Since $\Ind_{\Xi,!}$ is exact, the cochain complex $\Ind_{\Xi,!}(N^\bullet)$ is also acyclic, hence $\H^i(\Ind_{\Xi,!}(N^\bullet))=0$ for all $i$. So we get an isomorphism
$$\H^i(\Ind_{\Xi,!*}(N^\bullet))=\H^{i+1}(\J^{-1}_{\Xi,!}(N^\bullet)).$$
\end{remark}

From above discussion, we deduce the following useful Lemma.

\begin{lemma}
    \label{exact}
    Let $N^{-1}\xrightarrow{f}N^0\xrightarrow{g}N^1$ be an exact sequence of $\l_\Xi$-modules, then the induced sequence
    $$\Ind_{\Xi,!*}(N^{-1})\to\Ind_{\Xi,!*}(N^0)\to\Ind_{\Xi,!*}(N^1)$$
    is exact if the map $\J^{-1}_{\Xi,!}(g)\colon\J^{-1}_{\Xi,!}(N^0)\to\J^{-1}_{\Xi,!}(N^1)$ is surjective.
\end{lemma}

\begin{proof}
    We extend the exact sequence to an acyclic complex
    $$N^\bullet=[0\to\ker(f)\to N^{-1}\to N^0\to N^1\to\coker(g)\to0]$$
    such that $N^i$ lives in cohomological degree $i$. Then from Remark~\ref{exactness-minimal}, we have
    $$\H^0(\Ind_{\Xi,!*}(N^\bullet))=\H^1(\J^{-1}_{\Xi,!}(N^\bullet))=0.$$

    Therefore, we conclude that
    $$\Ind_{\Xi,!*}(N^{-1})\to\Ind_{\Xi,!*}(N^0)\to\Ind_{\Xi,!*}(N^1)$$
    is exact.
\end{proof}

\section{Weight modules}
\label{weight-module}

In this section, we focus on weight modules.

\subsection{Generalities on weight modules}

Let $\a$ be a finite-dimensional abelian Lie algebra.

\begin{definition}
    An $\a$-module $M$ is called a \textit{weight module} if
$$M=\bigoplus_{\lambda\in\a^*}M_\lambda,\text{ where }M_\lambda=\{v\in M:a\cdot v=\lambda(a)v\text{ for all }a\in\a\}.$$
    Let us denote by $\a\wtMod$ the category of weight $\a$-modules. It has a full subcategory $\a\wtMod^\fin$ consisting of weight $\a$-modules of which each weight space is finite dimensional.
\end{definition}

The following proposition is well known (cf. \cite{kac1990infinite} Proposition 1.5).

\begin{proposition}
    \label{weight-abelian}
    Any submodule or quotient of a $\a$-weight module is also a weight module.
\end{proposition}

\begin{proof}
    Let $M$ be a weight $\a$-module and $N\subset M$ be a submodule. For any $v\in N$, we can decompose $v$ as a finite sum $v=\sum_{j=1}^m v_j$, where $v_j\in M_{\lambda_j}$ and the weights $\lambda_1,\cdots,\lambda_m$ are distinct. The polynomial
    $$\prod_{1\leq i<j\leq m}(\lambda_i-\lambda_j)\in S\a^*=\C[\a]$$
    is nonzero, so we can find $a\in\a$ such that $\prod_{1\leq i<j\leq m}(\lambda_i-\lambda_j)(a)\neq0$. This means that $\lambda_1(a),\cdots,\lambda_m(a)$ are distinct.

    For $k=0,1,\cdots,m-1$, we have
    $$a^k\cdot v=\sum_{j=1}^m\lambda_j(a)^k v_j\in N.$$
    This is a system of linear equations associated to a nondegenerate matrix (the Vandermonde determinant does not vanish). Hence all $v_j$'s lie in $N$.

    Now let $f\colon M\proj L$ be a quotient. We know that $\ker f$ is a weight module, so the quotient $L$ is also a weight module.
\end{proof}

From Proposition~\ref{weight} above, we see that $\a\wtMod$ is an abelian category, and $\a\wtMod^\fin$ is a Serre subcategory.

\begin{definition}
    For an $\a$-weight module $M$, the \textit{support} of $M$ is
$$\supp(M):=\{\lambda\in\a^*:M_\lambda\neq0\}.$$
\end{definition}

\subsection{Weight Kac--Moody modules}
\label{km-weight}

Let $\g$ be a Kac--Moody algebra, with a fixed Cartan subalgebra $\h$. Let us fix a parabolic type $\Xi$ from now on in this section.

\begin{definition}
    We call a $\g$-module $M$ a \textit{weight module} if it is an $\h$-weight module. Categorically, the category of weight $\g$-modules, denoted by $\g\wtMod$, is the inverse image of $\h\wtMod$ under the forgetful functor $\g\Mod\xrightarrow{\operatorname{Oblv}^\g_\h}\h\Mod$. Moreover, we have the full subcategory $\g\wtMod^\fin$, consisting of weight $\g$-modules of which each weight space is finite dimensional, that is the inverse image of $\h\wtMod^\fin$ under $\operatorname{Oblv}^\g_\h$. We know that $\g\wtMod$ is an abelian category, with $\g\wtMod^\fin$ being a Serre subcategory.

    Similarly, we have the abelian category of weight $\l_\Xi$-modules $\l_\Xi\wtMod$ that is the inverse image of $\h\wtMod$ under the forgetful functor $\l_\Xi\Mod\xrightarrow{\operatorname{Oblv}^{\l_\Xi}_\h}\h\Mod$. It has a Serre subcategory $\l_\Xi\wtMod^\fin$ consisting of weight $\l_\Xi$-modules of which each weight space is finite dimensional.
\end{definition}

\begin{remark}
    From a homological algebra point of view, the category $\g\wtMod$ is a better object to study than $\g\wtMod^\fin$, because it has better homological properties. For example, it is a cocomplete symmetric monoidal category, while $\g\wtMod^\fin$ is not closed under arbitrary colimits and tensor products. Moreover, since $\g\wtMod^\fin$ is a Serre subcategory of $\g\wtMod$, many homological properties of $\g\wtMod^\fin$ can be recovered from those of $\g\wtMod$. For example, for two objects $N_1,N_2\in\g\wtMod^\fin$, the first extension group between them in $\g\wtMod^\fin$ is the same as the one in $\g\wtMod$.
\end{remark}

However, there is a novel construction on $\g\wtMod^\fin$ that cannot be performed on $\g\wtMod$, that is, the restricted duality functor presented below.

\begin{definition}
    Let $M\in\g\wtMod^\fin$ be a $\g$-weight module of which each weight space is finite dimensional. Let $M=\bigoplus_{\lambda\in\h^*}M_\lambda$ be the weight decomposition. The linear dual $M^*=\Hom(M,\C)$ is naturally a right $U\g$-module, can we make which into a left $U\g$-module via the transpose anti-involution $\tau$. Then the subspace
    $$M^\vee=\bigoplus_{\lambda\in\h^*}M_\lambda^*\subset M^*$$
    is a $\g$-submodule. We call $M^\vee$ the \textit{restricted dual} of $M$.
\end{definition}

The following properties are well-known (cf. \cite{humphreys2008representations}, \cite{kac1990infinite}).

\begin{proposition}
    \begin{enumerate}
        \item The operation of taking restricted dual defines a contravariant functor $$^\vee\colon\g\wtMod^\fin\to\g\wtMod^\fin.$$
        \item The functor $^\vee$ is an anti-involution, i.e. ${^\vee}{^\vee}=\id$.
        \item The restricted duality functor $^\vee$ is exact. Moreover, it defines an anti-equivalence of abelian categories.
        \item The simple highest weight modules are self-dual, i.e. $L(\lambda)^\vee\simeq L(\lambda)$ for each $\lambda\in\h^*$. 
    \end{enumerate}
\end{proposition}

As an easy corollary, we have

\begin{corollary}
    For any $N_1,N_2\in\g\wtMod^\fin\subset\g\wtMod$, there is a canonical isomorphism of Yoneda extension groups
    $$\Ext^n_{\g\wtMod}(N_1,N_2)=\Ext^n_{\g\wtMod}(N_2^\vee,N_1^\vee)$$
    for any $n\geq0$.
\end{corollary}

Repeating above constructions word-by-word, we can define the restricted duality functor $$^\vee\colon\l_\Xi\wtMod^\fin\to\l_\Xi\wtMod^\fin$$
satisfying the same properties.

\subsection{Weight modules under parabolic induction}

\begin{proposition}
    The functors $\Ind_{\Xi,!},\Ind_{\Xi,!*}$ and $\J^{-1}_{\Xi,!}$ restrict to functors
    $$\Ind_{\Xi,!}\colon\l_\Xi\wtMod\to\g\wtMod,$$
    $$\Ind_{\Xi,!*}\colon\l_\Xi\wtMod\to\g\wtMod,$$
    $$\J^{-1}_{\Xi,!}\colon\l_\Xi\wtMod\to\g\wtMod.$$
\end{proposition}

\begin{proof}
    For $N\in\l_\Xi\Mod$, the isomorphism $\Ind_{\Xi,!}(N)=U\u^-_\Xi\otimes N$ is not only an isomorphism of vector spaces, but it also preserves the $\h$-module structure. Here the right hand side is understood as the tensor product of $\h$-modules. Therefore, we see that if $N\in\l_\Xi\wtMod$, then $\Ind_{\Xi,!}(N)=U\u^-_\Xi\otimes N$ lives in $\g\wtMod$.

    As a quotient (resp. sub), we see that $\Ind_{\Xi,!*}$ (resp. $\J^{-1}_{\Xi,!}$) also maps weight modules to weight modules.
\end{proof}

\begin{remark}
    Unlike the parabolic $!$-induction $\Ind_{\Xi,!}$, the parabolic $*$-induction $\Ind_{\Xi,*}$ does not map a weight $\l_\Xi$-module to a weight $\g$-module in general (instead of direct sum of weight spaces, it is direct product of weight spaces). Though our construction in the previous section is completely ``dualizable'', asymmetric phenomena may happen when focusing on weight modules.
\end{remark}

\begin{proposition}
    The functors $\Res^!_\Xi,\Res^*_\Xi$ and $\Res^{!*}_\Xi$ restrict to functors
    $$\Res^!_\Xi\colon\g\wtMod\to\l_\Xi\wtMod,$$
    $$\Res^*_\Xi\colon\g\wtMod\to\l_\Xi\wtMod,$$
    $$\Res^{!*}_\Xi\colon\g\wtMod\to\l_\Xi\wtMod,$$
    $$\Res^!_\Xi\colon\g\wtMod^\fin\to\l_\Xi\wtMod^\fin,$$
    $$\Res^*_\Xi\colon\g\wtMod^\fin\to\l_\Xi\wtMod^\fin,$$
    $$\Res^{!*}_\Xi\colon\g\wtMod^\fin\to\l_\Xi\wtMod^\fin.$$
\end{proposition}

\begin{proof}
    The forgetful functor $\Oblv^\g_{\l_\Xi}\colon\g\Mod\to\l_\Xi\Mod$ preserves weight spaces. Therefore, as a sub (resp. quotient), $\Res^!_\Xi$ (resp. $\Res^*_\Xi$) can be restricted. Thus $\Res^{!*}_\Xi=\im(\Res^!_\Xi\to\Res^*_\Xi)$ can be restricted.
\end{proof}

\subsection{Simple weight modules}

\begin{definition}
    For $\xi\in\h^*$, an $\l_\Xi$-weight module $N\in\l_\Xi\wtMod$ is called \textit{$\xi$-shifted}, if $\supp(N)\subset\xi+Q_\Xi$.\footnote{Recall that $Q_\Xi$ is the root lattice of $\l_\Xi$.} Let us denote by $\l_\Xi\wtMod_\xi$ the full subcategory of $\l_\Xi\wtMod$ consisting of $\xi$-shifted weight modules.
\end{definition}

It is easy to see that we have a decomposition of abelian category:

\begin{proposition}
    If $\xi\equiv\xi'\pmod {Q_\Xi}$, then $\l_\Xi\wtMod_\xi=\l_\Xi\wtMod_{\xi'}$. There is a direct sum decomposition of abelian category
    $$\l_\Xi\wtMod=\bigoplus_{\xi\in\h^*/Q_\Xi}\l_\Xi\wtMod_\xi.$$
    Here $\xi$ runs through (a choice of representatives of) the coset space $\h^*/Q_\Xi$.
\end{proposition}

\begin{corollary}
    For any simple (hence indecomposable) weight $\l_\Xi$-module $N\in\l_\Xi\wtMod$, there exists $\xi\in\h^*$ such that $N\in\l_\Xi\wtMod_\xi$.
\end{corollary}

\begin{remark}
    \label{max-new}
    Let $N\in\l_\Xi\wtMod_\xi$ be a $\xi$-shifted $\l_\Xi$-weight module for some $\xi\in\h^*$. In this situation, for any $\g$-submodule $M$ of $\Ind_{\Xi,!}(N)=U\u^-_\Xi\otimes N$, we have, by Observation~\ref{weight}, that $M$ lies in $\u^-_\Xi(U\u^-_\Xi)\otimes N$ if and only if $M$ has zero intersection with $1\otimes N$. Therefore, in the sense of Remark~\ref{max}, $\J^{-1}_{\Xi,!}(N)$ is the maximal $\g$-submodule of $\Ind_{\Xi,!}(N)=U\u^-_\Xi\otimes N$ that has zero intersection with $1\otimes N$.
\end{remark}

\begin{proposition}
    \label{res-simple}
    The functors $\Res^!_\Xi,\Res^*_\Xi$ and $\Res^{!*}_\Xi$ map a simple $\g$-weight module to a simple $\l_\Xi$-weight module or $0$.
\end{proposition}

\begin{proof}
    Let $M$ be a simple $\g$-weight module. Suppose that $\Res^!_\Xi(M)\neq0$. Pick any nonzero weight vectors $v_1,v_2\in\Res^!_\Xi(M)\subset\Oblv^\g_{\l_\Xi}(M)$. By the simplicity of $M$, we have $M=U\g.v_1=U\g.v_2$. Noting that $v_i\in\Res^!_\Xi(M)=\Hom_{U\u^+_\Xi}(\C,M)$, we have $\u^+_\Xi.v_i=0$ for $i=1,2$. So by the PBW theorem,
    $$M=U\g.v_i=(U\u^-_\Xi\otimes U\l_\Xi\otimes U\u^+_\Xi).v_i=(U\u^-_\Xi\otimes U\l_\Xi).v_i.$$
    In particular, we can find $X_1,X_2\in U\u^-_\Xi\otimes U\l_\Xi$ such that $v_1=X_1.v_2,v_2=X_2.v_1$. Now we have $v_1=X_1X_2.v_1$, so the weight of $X_1X_2$ is $0$. By Observation~\ref{weight}, we see that $X_1,X_2\in1\otimes U\l_\Xi=U\l_\Xi$, so $v_1$ and $v_2$ lie in the same $\l_\Xi$-submodule of $\Res^!_\Xi(M)$. This verifies the simplicity of $\Res^!_\Xi(M)$.

    Dually, let $M'$ be a simple $\g$-weight module such that $\Res^*_\Xi(M')\neq0$. Pick any nonzero weight vectors $\bar{w}_1,\bar{w}_2\in\Res^*_\Xi(M')$. Let $w_i$ be a weight vector in $\Oblv^\g_{\l_\Xi}(M')$ that is in the preimage of $\bar{w}_i$. By the simplicity of $M'$, we have $M'=U\g.w_1=U\g.w_2$. Therefore, we can find $Y_1,Y_2\in U\g=U\u^-_\Xi\otimes U\l_\Xi\otimes U\u^+_\Xi$ such that $w_1=Y_1.w_2,w_2=Y_2.w_1$. Since $w_1,w_2$ do not vanish in the quotient $\Res^*_\Xi(M')$, we must have $Y_1,Y_2\in1\otimes U\l_\Xi\otimes U\u^+_\Xi$. Now we have $w_1=Y_1Y_2.w_1$, so the weight of $Y_1Y_2$ is $0$. By Observation~\ref{weight}, we see that $Y_1,Y_2\in U\l_\Xi\otimes 1=U\l_\Xi$, so $w_1,w_2$ (and hence $\bar{w}_1,\bar{w}_2$) lie in the same $\l_\Xi$-submodule. This verifies the simplicity of $\Res^*_\Xi(M')$. 
    
    The functor $\Res^{!*}_\Xi$ is a quotient of $\Res^!$ (and a sub of $\Res^*_\Xi$), so it also maps simple $\g$-weight modules to simple $\l_\Xi$-weight modules or $0$.
\end{proof}

\begin{corollary}
    For any weight $\lambda\in\h^*$, $\Res^!_\Xi(L(\lambda))=\Res^*_\Xi(L(\lambda))=\Res^{!*}_\Xi(L(\lambda))=L_\Xi(\lambda)$.
\end{corollary}

\begin{proof}
    Using Proposition~\ref{res-simple}, we conclude that $\Res^!_\Xi(L(\lambda)),\Res^*_\Xi(L(\lambda)),\Res^{!*}_\Xi(L(\lambda))$ are simple highest weight $\l_\Xi$-modules of highest weight $\lambda$. Therefore $\Res^!_\Xi(L(\lambda))=\Res^*_\Xi(L(\lambda))=\Res^{!*}_\Xi(L(\lambda))=L_\Xi(\lambda)$.
\end{proof}

\begin{proposition}
    \label{ind-simple}
    The functor $\Ind_{\Xi,!*}$ maps a simple $\l_\Xi$-weight module to a simple weight $\g$-module.
\end{proposition}

\begin{proof}
    Let $N$ be a simple $\l_\Xi$-weight module, then it is $\xi$-shifted for some $\xi$. Pick any nonzero $n\in N$, then $U\l_\Xi.n=N$ by the simplicity of $N$. Therefore $\Ind_{\Xi,!}(N)=U\g\otimes_{U\p^+_\Xi}N=U\g.(1\otimes n)$ and the image of $1\otimes n$ in $\Ind_{\Xi,!*}(N)$ is cyclic.
    
    Let $\bar{v}$ be any nonzero vector in $\Ind_{\Xi,!*}(N)$, and $v$ be a preimage of $\bar{v}$ in $\Ind_{\Xi,!}(N)$. Consider the $\g$-submodule $U\g.v\subset\Ind_{\Xi,!}(N)$. If it has zero intersection with $1\otimes N$, then by Remark~\ref{max-new}, it is contained in $\J^{-1}_{\Xi,!}(N)$. As a consequence, its image in $\Ind_{\Xi,!*}(N)$ will be zero. In particular, $\bar{v}=0$, a contradiction. Therefore the intersection $U\g.v\cap(1\otimes N)$ is nonzero. This means that there exists a nonzero element in $U\g.v$ of the form $1\otimes n$ for some $n\in N$. By previous discussion, we see that $1\otimes n$ generates $\Ind_{\Xi,!}(N)$. Hence $U\g.v=\Ind_{\Xi,!}(N)$ and $U\g.\bar{v}=\Ind_{\Xi,!*}(N)$. This shows the simplicity of $\Ind_{\Xi,!*}(N)$.
\end{proof}

\begin{remark}
    The above Proposition~\ref{ind-simple} justifies the term ``minimal''.
\end{remark}

\begin{corollary}
    \label{minimal-maps-hwsimple}
    For any weight $\lambda\in\h^*$, $\Ind_{\Xi,!*}(L_\Xi(\lambda))=L(\lambda)$.
\end{corollary}

\begin{proof}
    By Proposition~\ref{ind-simple}, $\Ind_{\Xi,!*}(L_\Xi(\lambda))$ is a simple $\g$-weight module. Noting that it is a highest weight module of highest weight $\lambda$, we conclude that $\Ind_{\Xi,!*}(L_\Xi(\lambda))=L(\lambda)$.
\end{proof}

\section{Applications}
\label{applications}

We exhibit two applications of our theory of minimal parabolic induction.

\subsection{Extensions}

In this section, we examine the behavior of first extension groups between some simple highest weight modules under minimal parabolic induction.

\begin{lemma}
    \label{res-exact}
    Let $\lambda\in\h^*$ be a weight. Suppose we have a short exact sequence in $\g\wtMod$
    $$0\to M_1\to M_2\to L(\lambda)\to0.$$
    Moreover, assume that there is no weight $\mu\in\supp(M_1)$ such that $\mu>\lambda$, then the restricted sequence
    $$0\to\Res^!_\Xi(M_1)\to\Res^!_\Xi(M_2)\to\Res^!_\Xi(L(\lambda))=L_\Xi(\lambda)\to0$$
    is also exact.
\end{lemma}

\begin{proof}
    The functor $\Res^!_\Xi$ is left exact, so we only need to show that $\Res^!_\Xi(M_2)\to L_\Xi(\lambda)$ is surjective. Let $v_\lambda$ be the highest weight vector in $L(\lambda)$. By assumption, $v_\lambda$ must come from a highest weight vector $v\in M_2$. We have $v\in M_2^{\n^+}\subset M_2^{\u^+_\Xi}=\Res^!_\Xi(M_2)$. Therefore $v_\lambda$, now viewed as the highest weight vector in $L_\Xi(\lambda)$, lives in the image of $\Res^!_\Xi(M_2)\to L_\Xi(\lambda)$. As a consequence, the map $\Res^!_\Xi(M_2)\to L_\Xi(\lambda)$ is surjective.
\end{proof}

\begin{proposition}
    \label{res-ext}
    Let $\mu,\lambda\in\h^*$ be two weights such that $\mu\equiv\lambda\pmod{Q_\Xi}$, then we have an inclusion
    $$\Ext^1_{\g\wtMod}(L(\lambda),L(\mu))\incl\Ext^1_{\l_\Xi\wtMod}(L_\Xi(\lambda),L_\Xi(\mu)).$$
\end{proposition}

\begin{proof}
    Using the duality functor introduced in Section~\ref{km-weight}, we have
    $$\Ext^1_{\g\wtMod}(L(\lambda),L(\mu))=\Ext^1_{\g\wtMod}(L(\mu),L(\lambda)),\Ext^1_{\l_\Xi\wtMod}(L_\Xi(\lambda),L_\Xi(\mu))=\Ext^1_{\l_\Xi\wtMod}(L_\Xi(\mu),L_\Xi(\lambda)).$$
    So without losing generality, we can assume that $\mu\ngtr\lambda$.

    Now by Lemma~\ref{res-exact}, for any extension
    $$[0\to L(\mu)\to P\to L(\lambda)\to0]\in\Ext^1_{\g\wtMod}(L(\lambda),L(\mu)),$$
    the restricted sequence
    $$0\to\Res^!_\Xi(L(\mu))=L_\Xi(\mu)\to\Res^!_\Xi(P)\to\Res^!_\Xi(L(\lambda))=L_\Xi(\lambda)\to0$$
    is also exact, and hence gives an element in $\Ext^1_{\l_\Xi\wtMod}(L_\Xi(\lambda),L_\Xi(\mu))$. Moreover, it induces a linear map
    $$\R\colon\Ext^1_{\g\wtMod}(L(\lambda),L(\mu))\to\Ext^1_{\l_\Xi\wtMod}(L_\Xi(\lambda),L_\Xi(\mu)).$$

    Let us show that $\R$ is injective. Let
    $$\left[0\to L(\mu)\xrightarrow{f} P\xrightarrow{g} L(\lambda)\to0\right]$$
    be an element in $\Ext^1_{\g\wtMod}(L(\lambda),L(\mu))$ such that
    $$\R\left[0\to L(\mu)\xrightarrow{f} P\xrightarrow{g} L(\lambda)\to0\right]=0.$$
    This means that the projection $\Res^!_\Xi(P)\proj\Res^!_\Xi(L(\lambda))=L_\Xi(\lambda)$ admits a section, which we denote by $\iota$. Denote by $v_\lambda$ the highest weight vector in $L(\lambda)$ (it is also the highest weight vector in $L_\Xi(\lambda)$), then any element in $L(\lambda)$ is of the form $X.v_\lambda$ for some $X\in U\g$. Moreover, we know that $\iota(v_\lambda)$ is a highest weight vector in $P$ because $\mu\ngtr\lambda$. We claim that the map $\Tilde{\iota}\colon L(\lambda)\to P$ defined by $X.v_\lambda\mapsto X.\iota(v_\lambda)$ is well-defined and provides a section for $g$.

    Suppose $X_1.v_\lambda=X_2.v_\lambda$ for some $X_1,X_2\in U\g$, then
    $$g((X_1-X_2).\iota(v_\lambda))=(X_1-X_2).g(\iota(v_\lambda))=(X_1-X_2).v_\lambda=0.$$
    Therefore $(X_1-X_2).\iota(v_\lambda)=f(v)$ for some $v\in L(\mu)$. If $v\neq0$, then it generates $L(\mu)$ by the simplicity of $L(\mu)$. In particular, there exists $Y\in U\g$ such that $Y.v=v_\mu$, where $v_\mu$ is the highest weight vector of $L(\mu)$ (it is also the highest weight vector of $L_\Xi(\mu)$). Now we have
    $$f(v_\mu)=Y.f(v)=Y(X_1-X_2).\iota(v_\lambda)\subset U\g.\iota(v_\lambda)=(U\u^-_\Xi\otimes U\l_\Xi\otimes U\u^+_\Xi).\iota(v_\lambda)=(U\u^-_\Xi\otimes U\l_\Xi).\iota(v_\lambda).$$
    Since $\lambda\equiv\mu\pmod{Q_\Xi}$, we have, by Observation~\ref{weight},
    $$f(v_\mu)\in U\l_\Xi.v_\lambda=\iota(L_\Xi(\lambda)).$$
    This contradicts to the fact that $\iota$ is a section of $\Res^!_\Xi(g)$. Therefore, we have $v=0$. Hence $X_1.\iota(v_\lambda)=X_2.\iota(v_\lambda)$, and $\Tilde{\iota}$ is well-defined. Now it is transparent to see that $\Tilde{\iota}$ provides a section for $g$.
\end{proof}

We introduce the following definition for technical reason.

\begin{definition}
    \label{joyful}
    An integral weight $\lambda\in\h^*$ is called \textit{$\Xi$-joyful}, if
    $$N(\lambda)=\J^{-1}_{\Xi,!}(M_\Xi(\lambda))+\Ind_{\Xi,!}(N_\Xi(\lambda)).$$
    Recall that $N(\lambda)$ (resp. $N_\Xi(\lambda)$) is the maximal $\g$ (resp. $\l_\Xi$) proper submodule of $M(\lambda)$ (resp. $M_\Xi(\lambda)$).
\end{definition}

\begin{remark}
    We always have an inclusion
    $$\J^{-1}_{\Xi,!}(M_\Xi(\lambda))+\Ind_{\Xi,!}(N_\Xi(\lambda))\incl N(\lambda).$$
\end{remark}

\begin{proposition}
    \label{gen-singular}
    Suppose $\lambda$ is an integral weight and $N(\lambda)$ is generated by singular vectors, i.e.
    $$N(\lambda)=\sum_{\alpha\in\Pi:s_\alpha\cdot\lambda<\lambda}M(s_\alpha\cdot\lambda),$$
    then $\lambda$ is $\Xi$-joyful for any parabolic type $\Xi$.
\end{proposition}

\begin{proof}
    We have
    $$\sum_{\alpha\in\Pi\backslash\Xi:s_\alpha\cdot\lambda<\lambda}M(s_\alpha\cdot\lambda)\subset\u^-_\Xi(U\u^-_\Xi)\otimes M_\Xi(\lambda),$$
    so
    $$\sum_{\alpha\in\Pi\backslash\Xi:s_\alpha\cdot\lambda<\lambda}M(s_\alpha\cdot\lambda)\subset\J^{-1}_{\Xi,!}(M_\Xi(\lambda))$$
    by Proposition~\ref{j-2}.

    Moreover,
    $$\sum_{\alpha\in\Xi:s_\alpha\cdot\lambda<\lambda}M(s_\alpha\cdot\lambda)=\Ind_{\Xi,!}\left(\sum_{\alpha\in\Xi:s_\alpha\cdot\lambda<\lambda}M_\Xi(s_\alpha\cdot\lambda)\right)\subset\Ind_{\Xi,!}(N_\Xi(\lambda)),$$
    so
    $$N(\lambda)=\sum_{\alpha\in\Pi\backslash\Xi:s_\alpha\cdot\lambda<\lambda}M(s_\alpha\cdot\lambda)+\sum_{\alpha\in\Xi:s_\alpha\cdot\lambda<\lambda}M(s_\alpha\cdot\lambda)\subset\J^{-1}_{\Xi,!}(M_\Xi(\lambda))+\Ind_{\Xi,!}(N_\Xi(\lambda)).$$
\end{proof}

The notion of $\Xi$-joyful is useful, due to the following Proposition.

\begin{proposition}
    \label{sing-surj}
    Suppose an integral weight $\lambda\in\h^*$ is $\Xi$-joyful, then the map
    $$\J^{-1}_{\Xi,!}(M_\Xi(\lambda))\to\J^{-1}_{\Xi,!}(L_\Xi(\lambda))$$
    induced from the projection $M_\Xi(\lambda)\proj L_\Xi(\lambda)$ is surjective.
\end{proposition}

\begin{proof}
    Let us consider a commutative diagram with exact rows
    $$\begin{tikzcd}
0 \arrow[r] & {\J^{-1}_{\Xi,!}(M_\Xi(\lambda))} \arrow[r] \arrow[d, "a"] & {\Ind_{\Xi,!}(M_\Xi(\lambda))} \arrow[r] \arrow[d, "b"] & {\Ind_{\Xi,!*}(M_\Xi(\lambda))} \arrow[r] \arrow[d, "c"] & 0 \\
0 \arrow[r] & {\J^{-1}_{\Xi,!}(L_\Xi(\lambda))} \arrow[r]           & {\Ind_{\Xi,!}(L_\Xi(\lambda))} \arrow[r]           & {\Ind_{\Xi,!*}(L_\Xi(\lambda))} \arrow[r]           & 0
\end{tikzcd}$$
By the Snake Lemma, it suffices to show that the map
$$\ker(\Ind_{\Xi,!}(M_\Xi(\lambda))\xrightarrow{b}\Ind_{\Xi,!}(L_\Xi(\lambda)))\to\ker(\Ind_{\Xi,!*}(M_\Xi(\lambda))\xrightarrow{c}\Ind_{\Xi,!*}(L_\Xi(\lambda)))$$
is surjective. Since $\Ind_{\Xi,!}$ is exact,
$$\ker(\Ind_{\Xi,!}(M_\Xi(\lambda))\xrightarrow{b}\Ind_{\Xi,!}(L_\Xi(\lambda)))=\Ind_{\Xi,!}(N_\Xi(\lambda)).$$
Notice that $\Ind_{\Xi,!}(M_\Xi(\lambda))=M(\lambda)$ and $\Ind_{\Xi,!*}(L_\Xi(\lambda))=L(\lambda)$,
$$\ker(\Ind_{\Xi,!}(M_\Xi(\lambda))\to\Ind_{\Xi,!*}(L_\Xi(\lambda)))=\ker(M(\lambda)\to L(\lambda))=N(\lambda).$$
Therefore
$$\ker(\Ind_{\Xi,!*}(M_\Xi(\lambda))\xrightarrow{c}\Ind_{\Xi,!*}(L_\Xi(\lambda)))=\frac{\ker(\Ind_{\Xi,!}(M_\Xi(\lambda))\to\Ind_{\Xi,!*}(L_\Xi(\lambda)))}{\ker(\Ind_{\Xi,!}(M_\Xi(\lambda))\to\Ind_{\Xi,!*}(M_\Xi(\lambda)))}=\frac{N(\lambda)}{\J^{-1}_{\Xi,!}(M_\Xi(\lambda))}.$$
By our assumption that $\lambda$ is $\Xi$-joyful, the map $\Ind_{\Xi,!}(N_\Xi(\lambda))\to\frac{N(\lambda)}{\J^{-1}_{\Xi,!}(M_\Xi(\lambda))}$ is surjective, we are done.
\end{proof}

\begin{lemma}
    \label{ind-exact}
    Let $\mu,\lambda\in\h^*$ be two integral weights such that $\mu-\lambda\notin\Z_{\geq0}\Xi$. Moreover, suppose that $\lambda$ is $\Xi$-joyful, then for any short exact sequence of weight $\l_\Xi$-modules
    $$0\to L_\Xi(\mu)\to Q\to L_\Xi(\lambda)\to0,$$
    the induced sequence
    $$0\to\Ind_{\Xi,!*}(L_\Xi(\mu))=L(\mu)\to\Ind_{\Xi,!*}(Q)\to\Ind_{\Xi,!*}(L_\Xi(\lambda))=L(\lambda)\to0$$
    is exact.
\end{lemma}

\begin{proof}
    By Proposition~\ref{inj-surj}, the map $\Ind_{\Xi,!*}(L_\Xi(\mu))\to\Ind_{\Xi,!*}(Q)$ is injective and the map $\Ind_{\Xi,!*}(Q)\to\Ind_{\Xi,!*}(L_\Xi(\lambda))$ is surjective. It remains to check the exactness of the middle term.

    Since $\mu-\lambda\notin\Z_{\geq0}\Xi$, the highest weight vector in $L_\Xi(\lambda)$ must come from a highest weight vector in $Q$. Therefore we have an $\l_\Xi$-module morphism $M_\Xi(\lambda)\to Q$ such that the composition $M_\Xi(\lambda)\to Q\to L_\Xi(\lambda)$ is nonzero. This induces a sequence
    $$\J^{-1}_{\Xi,!}(M_\Xi(\lambda))\to\J^{-1}_{\Xi,!}(Q)\to\J^{-1}_{\Xi,!}(L_\Xi(\lambda)).$$
    The composition $\J^{-1}_{\Xi,!}(M_\Xi(\lambda))\to\J^{-1}_{\Xi,!}(L_\Xi(\lambda))$ is surjective by Proposition~\ref{sing-surj}, so the map $\J^{-1}_{\Xi,!}(Q)\to\J^{-1}_{\Xi,!}(L_\Xi(\lambda))$ is surjective. Now using Lemma~\ref{exact}, we conclude that the middle term
    $$\Ind_{\Xi,!*}(L_\Xi(\mu))\to\Ind_{\Xi,!*}(Q)\to\Ind_{\Xi,!*}(L_\Xi(\lambda))$$
    is exact.
\end{proof}

Now we can state and prove the main result of the section.

\begin{proposition}
    \label{ind-ext}
    Let $\mu,\lambda\in\h^*$ be two integral weights such that $\mu\equiv\lambda\pmod{Q_\Xi}$ and $\mu-\lambda\notin\Z_{\geq0}\Xi$. Moreover, suppose that $\lambda$ is $\Xi$-joyful, then we have isomorphisms
    $$\Ext^1_{\g\wtMod}(L(\lambda),L(\mu))=\Ext^1_{\g\wtMod}(L(\mu),L(\lambda))=\Ext^1_{\l_\Xi\wtMod}(L_\Xi(\lambda),L_\Xi(\mu))=\Ext^1_{\l_\Xi\wtMod}(L_\Xi(\mu),L_\Xi(\lambda)).$$
\end{proposition}

\begin{proof}
    Using the restricted dual introduced in Section~\ref{km-weight}, it is enough to show that
    $$\Ext^1_{\g\wtMod}(L(\lambda),L(\mu))=\Ext^1_{\l_\Xi\wtMod}(L_\Xi(\lambda),L_\Xi(\mu)).$$
    In Proposition~\ref{res-ext}, we have constructed an inclusion $$\R\colon\Ext^1_{\g\wtMod}(L(\lambda),L(\mu))\incl\Ext^1_{\l_\Xi\wtMod}(L_\Xi(\lambda),L_\Xi(\mu))$$
    that is induced by the functor $\Res^!_\Xi$.

    By Lemma~\ref{ind-exact}, for any element
    $$[0\to L_\Xi(\mu)\to Q\to L_\Xi(\lambda)\to0]\in\Ext^1_{\l_\Xi\wtMod}(L_\Xi(\lambda),L_\Xi(\mu)),$$
    the induced sequence
    $$0\to\Ind_{\Xi,!*}(L_\Xi(\mu))=L(\mu)\to\Ind_{\Xi,!*}(Q)\to\Ind_{\Xi,!*}(L_\Xi(\lambda))=L(\lambda)\to0$$
    is exact, and hence gives an element in $\Ext^1_{\g\wtMod}(L(\lambda),L(\mu))$. This operation induces a linear map
    $$\I\colon\Ext^1_{\l_\Xi\wtMod}(L_\Xi(\lambda),L_\Xi(\mu))\to\Ext^1_{\g\wtMod}(L(\lambda),L(\mu)).$$
    By Proposition~\ref{res-ind}, $\R\circ\I=\id$. Combining with the fact that $\R$ is injective, we conclude that $\R$ induces an isomorphism $\R\colon\Ext^1_{\g\wtMod}(L(\lambda),L(\mu))\xrightarrow{\sim}\Ext^1_{\l_\Xi\wtMod}(L_\Xi(\lambda),L_\Xi(\mu))$.
\end{proof}

\subsection{Annihilators}

In this section, let us assume $\Xi$ is of finite type, i.e. $\l_\Xi$ is a finite dimensional reductive Lie algebra.

Let $\lambda\in\h^*$ be a $\rho$-anti-dominant integral weight for $\l_\Xi$, i.e. $\langle\lambda+\rho,\alpha^\vee\rangle\in\Z_{\leq0}$ for any $\alpha^\vee\in\Xi^\vee$. The following statement is a direct consequence of Duflo's theorem on annihilators of Verma modules \cite{duflo1973construction}.

\begin{proposition}
    Let $\lambda\in\h^*$ be a $\rho$-anti-dominant integral weight for $\l_\Xi$, then for any $w\in W_\Xi$, we have
    $$\Ann_{U\l_\Xi}(M_\Xi(\lambda))=\Ann_{U\l_\Xi}(M_\Xi(w\cdot\lambda)).$$
\end{proposition}

We prove that this property is stable under minimal parabolic induction.

\begin{proposition}
    \label{ann}
    Let $\lambda\in\h^*$ be a $\rho$-anti-dominant integral weight for $\l_\Xi$, then for any $w\in W_\Xi$, we have
    $$\Ann_{U\g}(\Ind_{\Xi,!*}(M_\Xi(\lambda)))=\Ann_{U\g}(\Ind_{\Xi,!*}(M_\Xi(w\cdot\lambda))).$$
\end{proposition}

\begin{proof}
    Since $\lambda$ is a $\rho$-anti-dominant integral weight for $\l_\Xi$, $M_\Xi(\lambda)=L_\Xi(\lambda)$ is simple, and $L_\Xi(\lambda)$ is the unique simple submodule of $M_\Xi(w\cdot\lambda)$ (its socle). The inclusion $$M_\Xi(\lambda)=L_\Xi(\lambda)\incl M_\Xi(w\cdot\lambda)$$
    induces an inclusion
    $$\Ind_{\Xi,!*}(M_\Xi(\lambda))\incl\Ind_{\Xi,!*}(M_\Xi(w\cdot\lambda)).$$
    Therefore $\Ann_{U\g}(\Ind_{\Xi,!*}(M_\Xi(\lambda)))\supset\Ann_{U\g}(\Ind_{\Xi,!*}(M_\Xi(w\cdot\lambda)))$. Let $I=\Ann_{U\g}(\Ind_{\Xi,!*}(M_\Xi(\lambda)))$, $I$ is a two-sided ideal of $U\g$. It remains to prove that $I.\Ind_{\Xi,!*}(M_\Xi(w\cdot\lambda))=0$.

    Suppose $I.\Ind_{\Xi,!*}(M_\Xi(w\cdot\lambda))\neq0$, pick any nonzero $\bar{v}\in I.\Ind_{\Xi,!*}(M_\Xi(w\cdot\lambda))$. Let $v$ be a preimage of $\bar{v}$ in $\Ind_{\Xi,!}(M_\Xi(w\cdot\lambda))=U\u^-_\Xi\otimes M_\Xi(\lambda)$. Since $\bar{v}$ is nonzero in $\Ind_{\Xi,!*}(M_\Xi(w\cdot\lambda))$, $U\g.v$ has nonzero intersection with $1\otimes M_\Xi(w\cdot\lambda)$ (Remark~\ref{max-new}). Noting that $M_\Xi(\lambda)$ is the only simple submodule of $M_\Xi(w\cdot\lambda)$, we have
    $$1\otimes M_\Xi(\lambda)\subset(U\g.v)\cap(1\otimes M_\Xi(w\cdot\lambda)).$$

    Let $v_\lambda$ (resp. $v_{w\cdot\lambda}$) be the highest weight vector of $M_\Xi(\lambda)$ (resp. $M_\Xi(w\cdot\lambda)$), $\bar{v}_\lambda$ (resp. $\bar{v}_{w\cdot\lambda}$) be the image of $1\otimes v_\lambda$ (resp. $1\otimes v_{w\cdot\lambda}$) in $\Ind_{\Xi,!*}(M_\Xi(\lambda))$ (resp. $\Ind_{\Xi,!*}(M_\Xi(w\cdot\lambda))$), then
    $$\bar{v}_\lambda\in U\g.\bar{v}\subset I.\Ind_{\Xi,!*}(M_\Xi(w\cdot\lambda))=IU\g.\bar{v}_{w\cdot\lambda}=I.\bar{v}_{w\cdot\lambda}.$$

    Choose $X\in I\subset U\g=U\u^-_\Xi\otimes U\l_\Xi\otimes U\u^+_\Xi$ such that $\bar{v}_\lambda=X.\bar{v}_{w\cdot\lambda}$. Let us decompose $X$ as $X=X_1+X_2+X_3$ for some
    $$X_1\in1\otimes U\l_\Xi\otimes 1=U\l_\Xi,X_2\in\u^-_\Xi(U\u^-_\Xi)\otimes U\l_\Xi\otimes1,X_3\in U\u^-_\Xi\otimes U\l_\Xi\otimes\u^+_\Xi(U\u^+_\Xi).$$
    Since $\bar{v}_{w\cdot\lambda}$ is a highest weight vector, $X_3.\bar{v}_{w\cdot\lambda}=0$, so
    $$\bar{v}_\lambda=X_1.\bar{v}_{w\cdot\lambda}+X_2.\bar{v}_{w\cdot\lambda}.$$
    Noticing that $w\cdot\lambda-\lambda\in Q_\Xi$, we have, by Observation~\ref{weight}, $\bar{v}_\lambda=X_1.\bar{v}_{w\cdot\lambda}$. But $X\in I=\Ann_{U\g}(\Ind_{\Xi,!*}(M_\Xi(\lambda)))$, for any $w\in M_\Xi(\lambda)$,
    $$0=X.\bar{w}=X_1.\bar{w}+X_2.\bar{w}.$$
    Here $\bar{w}$ means the image of $1\otimes w$ in $\Ind_{\Xi,!*}(M_\Xi(\lambda))$. Therefore
    $$1\otimes(X_1.w)+X_2.(1\otimes w)\in\J^{-1}_{\Xi,!}(M_\Xi(\lambda))\subset\u^-_\Xi(U\u^-_\Xi)\otimes M_\Xi(\lambda).$$
    Thus $X_1.w=0$. Consequently,
    $$X_1\in\Ann_{U\l_\Xi}(M_\Xi(\lambda))=\Ann_{U\l_\Xi}(M_\Xi(w\cdot\lambda)).$$
    This implies that $\bar{v}_\lambda=X_1.\bar{v}_{w\cdot\lambda}=0$, which is absurd. So $I$ annihilates $\Ind_{\Xi,!*}(M_\Xi(w\cdot\lambda))$, we are done.
\end{proof}

\begin{corollary}
    Let $\lambda\in\h^*$ be a $\rho$-anti-dominant integral weight for $\l_\Xi$, then for any $w\in W_\Xi,$ we have
    $$\Ann_{U\g}(L(\lambda))\subset\Ann_{U\g}(L(w\cdot\lambda)).$$
\end{corollary}

\begin{proof}
    Since $\lambda$ is a $\rho$-anti-dominant integral weight for $\l_\Xi$, $L_\Xi(\lambda)=M_\Xi(\lambda)$. By Corollary~\ref{minimal-maps-hwsimple},
    $$L(\lambda)=\Ind_{\Xi,!*}(L_\Xi(\lambda))=\Ind_{\Xi,!*}(M_\Xi(\lambda)).$$
    From Proposition~\ref{ann}, we know that
    $$\Ann_{U\g}(L(\lambda))=\Ann_{U\g}(\Ind_{\Xi,!*}(M_\Xi(\lambda)))=\Ann_{U\g}(\Ind_{\Xi,!*}(M_\Xi(w\cdot\lambda))).$$
    The quotient $M_\Xi(w\cdot\lambda)\proj L_\Xi(w\cdot\lambda)$ induces a surjection $\Ind_{\Xi,!*}(M_\Xi(w\cdot\lambda))\proj\Ind_{\Xi,!*}(L_\Xi(w\cdot\lambda))=L(w\cdot\lambda)$ (Proposition~\ref{inj-surj}), so
    $$\Ann_{U\g}(L(\lambda))=\Ann_{U\g}(\Ind_{\Xi,!*}(M_\Xi(w\cdot\lambda)))\subset\Ann_{U\g}(L(w\cdot\lambda)).$$
\end{proof}

\section{Minimal type}

To give a hopefully more accessible orientation to our general theory, let us consider the minimal parabolic type, i.e. $\Xi=\{\alpha\}$ consists of a single simple root $\alpha$. For brevity, let us write $\p^+_\alpha$ for $\p^+_{\{\alpha\}}$, $\Ind_{\alpha,!*}$ for $\Ind_{\{\alpha\},!*}$, etc. In this case, $\l_\alpha$ is a finite dimensional reductive Lie algebra, whose derived subalgebra $\s_\alpha=[\l_\alpha,\l_\alpha]$ is isomorphic to $\sl_2$. Let $\z_\alpha$ be the center of $\l_\alpha$, then $\z_\alpha\in\h$ and $\l_\alpha=\s_\alpha\oplus\z_\alpha$.

For any $n\in\Z_{\geq0}$ and $\xi\in\z_\alpha^*\subset\h^*$, we have a non-split short exact sequence
$$0\to M_\alpha\left(-\frac{n+2}{2}\alpha+\xi\right)\to M_\alpha\left(\frac{n}{2}\alpha+\xi\right)\to L_\alpha\left(\frac{n}{2}\alpha+\xi\right)\to0.$$
Notice that $-\frac{n+2}{2}\alpha+\xi$ is a $\rho$-anti-dominant integral weight for $\l_\alpha$, so $M_\alpha\left(-\frac{n+2}{2}\alpha+\xi\right)=L_\alpha\left(-\frac{n+2}{2}\alpha+\xi\right)$ is simple. By applying minimal parabolic induction $\Ind_{\alpha,!*}$, we obtain a sequence
$$0\to L\left(-\frac{n+2}{2}\alpha+\xi\right)\to\Ind_{\alpha,!*}\left(M_\alpha\left(\frac{n}{2}\alpha+\xi\right)\right)\to L\left(\frac{n}{2}\alpha+\xi\right)\to0.$$

By Proposition~\ref{ann}, we have
$$\Ann_{U\g}\left(L\left(-\frac{n+2}{2}\alpha+\xi\right)\right)=\Ann_{U\g}\left(\Ind_{\alpha,!*}\left(M_\alpha\left(\frac{n}{2}\alpha+\xi\right)\right)\right)\subset\Ann_{U\g}\left(L\left(\frac{n}{2}\alpha+\xi\right)\right)$$
in the above sequence.

Suppose, in addition, that the weight $\frac{n}{2}\alpha+\xi$ is integral and is $\alpha$-joyful, then by Proposition~\ref{ind-ext}, the above sequence is exact. It does not split, because after precomposing with $\Res^!_\alpha$, it does not split. This gives rise to a nonzero element in $\Ext^1_{\g\wtMod}\left(L\left(\frac{n}{2}\alpha+\xi\right),L\left(-\frac{n+2}{2}\alpha+\xi\right)\right)$.

\begin{example}
    Let us consider $\g=\widehat{E_8}$ to be the untwited affine Kac--Moody algebra associated to $E_8$. Let $\Pi=\{\alpha_0,\alpha_1,\cdots,\alpha_8\}$ be the set of simple roots of $\widehat{E_8}$, where $\alpha_1,\cdots,\alpha_8$ are those finite roots of finite $E_8$. Here we are referring to the Bourbaki convention on root systems \cite{bourbakigroupes}.
    
    Recall that in Section~\ref{transpose-anti-involution}, we have introduced the Chevalley generators $e_i\in\g_{\alpha_i}$, $f_i\in\g_{-\alpha_i}$. Let $\Lambda_0,\cdots,\Lambda_8$ be the fundamental weights of $\widehat{E_8}$. Let $s_i$ be the simple reflection in $W$ associated to the simple root $\alpha_i$.

    Notice that $-\Lambda_4=s_2\cdot(-\Lambda_4)+\alpha_2$. As described above, we have a sequence
    \begin{equation}
        \label{sequence}
        0\to L(s_2\cdot(-\Lambda_4))\to\Ind_{\alpha_2,!*}(M_{\alpha_2}(-\Lambda_4))\to L(-\Lambda_4)\to0,
    \end{equation}
    where
    $$\Ann_{U\widehat{E_8}}(L(s_2\cdot(-\Lambda_4)))=\Ann_{U\widehat{E_8}}(\Ind_{\alpha_2,!*}(M_{\alpha_2}(-\Lambda_4)))\subset\Ann_{U\widehat{E_8}}(L(-\Lambda_4)).$$

    In this case, we can use a tricky way to show that the sequence~\eqref{sequence} is exact. Let $L_{-6}(E_8)$ be the simple affine vertex algebra associated to $E_8$ of level $-6$. By the classification of \cite{arakawa2018joseph}, we know that $L(s_2\cdot(-\Lambda_4))$ is an $L_{-6}(E_8)$-module. Since $\Ann_{U\widehat{E_8}}(L(s_2\cdot(-\Lambda_4)))=\Ann_{U\widehat{E_8}}(\Ind_{\alpha_2,!*}(M_{\alpha_2}(-\Lambda_4)))$, $\Ind_{\alpha_2,!*}(M_{\alpha_2}(-\Lambda_4))$ is also an $L_{-6}(E_8)$-module.\footnote{Recall that $L_{-6}(E_8)=V^{-6}(E_8)/J$. Here $V^{-6}(E_8)$ is the universal affine vertex algebra associated to $E_8$ of level $-6$, $J$ is some vertex algebra ideal. A $V^{-6}(E_8)$-module $M$ is an $L_{-6}(E_8)$-module if and only if $Y(v,z).m=0$ for any $v\in J$ and $m\in M$, where $Y(v,z)=\sum_{n\in\Z}v_{(n)}z^{-n-1}$ is the field corresponding to $v$. Notice that we can write the Fourier modes $v_{(n)}$ as elements from $U\widehat{E_8}$. More precisely, $v_{(n)}\in\End(M)$ lies in the image of $U\widehat{E_8}\to\End(M)$ for any $n\in\Z$. Therefore, the $V^{-6}(E_8)$-module $M$ is an $L_{-6}(E_8)$-module if and only if $v_{(n)}\in\Ann_{U\widehat{E_8}}(M)$ for any $v\in J$.} Again using the classification in \cite{arakawa2018joseph}, we see that each composition factor of $\Ind_{\alpha_2,!*}(M_{\alpha_2}(-\Lambda_4))$ is of the form $L(w\cdot(-\Lambda_4))$ for some
    $$w\in\{\id,s_2,s_3,s_1s_3,s_5,s_6s_5,s_7s_6s_5,s_8s_7s_6s_5,s_0s_8s_7s_6s_5\}.$$
    We have
    \begin{align*}
        -\Lambda_4&=s_2\cdot(-\Lambda_4)+\alpha_2,\\
        -\Lambda_4&=s_3\cdot(-\Lambda_4)+\alpha_3,\\
        s_3\cdot(-\Lambda_4)&=s_1s_3\cdot(-\Lambda_4)+2\alpha_1,\\
        -\Lambda_4&=s_5\cdot(-\Lambda_4)+\alpha_5,\\
        s_5\cdot(-\Lambda_4)&=s_6s_5\cdot(-\Lambda_4)+2\alpha_6,\\
        s_6s_5\cdot(-\Lambda_4)&=s_7s_6s_5\cdot(-\Lambda_4)+3\alpha_7,\\
        s_7s_6s_5\cdot(-\Lambda_4)&=s_8s_7s_6s_5\cdot(-\Lambda_4)+4\alpha_8,\\
        s_8s_7s_6s_5\cdot(-\Lambda_4)&=s_0s_8s_7s_6s_5\cdot(-\Lambda_4)+5\alpha_0.
    \end{align*}
    Let $v_{-\Lambda_4}$ be the highest weight vector in $M_{\alpha_2}(-\Lambda_4)$, and $\bar{v}_{-\Lambda_4}$ be the image of $1\otimes v_{-\Lambda_4}$ in $\Ind_{\alpha_2,!*}(M_{\alpha_2}(-\Lambda_4))$. For $i\in\{0,1,3,5,6,7,8\}$, we have
    $$e_jf_i.(1\otimes v_{-\Lambda_4})=f_ie_j.(1\otimes v_{-\Lambda_4})=0.$$
    for $j\in\{0,\cdots,8\}\backslash\{i\}$. Moreover,
    $$e_if_i.(1\otimes v_{-\Lambda_4})=f_ie_i.(1\otimes v_{-\Lambda_4})+\alpha_i^\vee.(1\otimes v_{-\Lambda_4})=0+\langle-\Lambda_4,\alpha_i^\vee\rangle(1\otimes v_{-\Lambda_4})=0.$$
    Therefore, $f_i.(1\otimes v_{-\Lambda_4})$ is a (possibly zero) highest weight vector in $\Ind_{\alpha_2,!}(M_{\alpha_2}(-\Lambda_4))=U\u^-_{\alpha_2}\otimes M_{\alpha_2}(-\Lambda_4)$. In particular, it generates an $\widehat{E_8}$-submodule of $\Ind_{\alpha_2,!}(M_{\alpha_2}(-\Lambda_4))=U\u^-_{\alpha_2}\otimes M_{\alpha_2}(-\Lambda_4)$ that has zero intersection with $1\otimes M_{\alpha_2}(-\Lambda_4)$. By Remark~\ref{max-new}, we see that this submodule lies in $\J^{-1}_{\alpha_2,!}(M_{\alpha_2}(-\Lambda_4))$, and hence $f_i.\bar{v}_{-\Lambda_4}=0$ in $\Ind_{\alpha_2,!*}(M_{\alpha_2}(-\Lambda_4))$. This shows that $L(w\cdot(-\Lambda_4))$ cannot appear as a composition factor of $\Ind_{\alpha_2,!*}(M_{\alpha_2}(-\Lambda_4))$ for $w\in\{s_3,s_1s_3,s_5,s_6s_5,s_7s_6s_5,s_8s_7s_6s_5,s_0s_8s_7s_6s_5\}$, and hence the only possible composition factors of $\Ind_{\alpha_2,!*}(M_{\alpha_2}(-\Lambda_4))$ are $L(-\Lambda_4)$ and $L(s_2\cdot(-\Lambda_4))$. Notice that the weight spaces $\Ind_{\alpha_2,!}(M_{\alpha_2}(-\Lambda_4))_{-\Lambda_4}$ and $\Ind_{\alpha_2,!}(M_{\alpha_2}(-\Lambda_4))_{s_2\cdot(-\Lambda_4)}$ are all one dimensional, so both $L(-\Lambda_4)$ and $L(s_2\cdot(-\Lambda_4))$ can appear at most once as composition factors of $\Ind_{\alpha_2,!*}(M_{\alpha_2}(-\Lambda_4))$. This shows that exactness of the sequence~\eqref{sequence}. Moreover, we see that this is a short exact sequence of $L_{-6}(E_8)$-modules.
    
    By similar careful analysis, we see that the extension quiver of simple highest weight $L_{-6}(E_8)$-modules is the same as the double quiver of the affine Dynkin quiver of type $E_8^{(1)}$.

    \begin{theorem}
        Consider the diagram
        $$
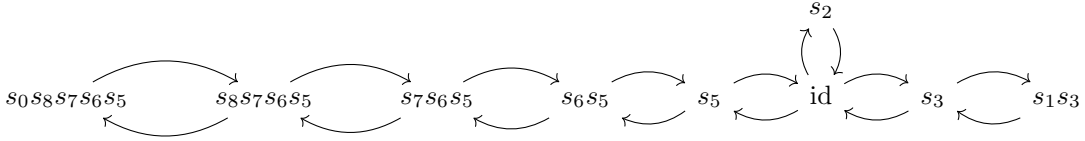
\begin{tikzcd}
 &        &        &            &                                               & s_2 \arrow[d, bend left]                                           &                                               &                             \\
s_0s_8s_7s_6s_5 \arrow[r, bend left] & s_8s_7s_6s_5 \arrow[r, bend left] \arrow[l, bend left] & s_7s_6s_5 \arrow[r, bend left] \arrow[l, bend left] & s_6s_5 \arrow[r, bend left] \arrow[l, bend left] & s_5 \arrow[r, bend left] \arrow[l, bend left] & \id \arrow[r, bend left] \arrow[l, bend left] \arrow[u, bend left] & s_3 \arrow[r, bend left] \arrow[l, bend left] & s_1s_3 \arrow[l, bend left]
\end{tikzcd}$$
    Then we have
    $$\dim\Ext^1_{\widehat{E_8}\wtMod}(L(w\cdot(-\Lambda_4)),L(w'\cdot(-\Lambda_4)))=\#\text{ of arrows from }w\text{ to }w'\text{ in the above diagram}$$
    for $w,w'\in\{\id,s_2,s_3,s_1s_3,s_5,s_6s_5,s_7s_6s_5,s_8s_7s_6s_5,s_0s_8s_7s_6s_5\}$.

    Furthermore, all nontrivial extensions are in fact extensions of $L_{-6}(E_8)$-modules.
    \end{theorem}

    We can perform similar constructions for vertex algebras $L_{-4}(E_7),L_{-3}(E_6)$ and $L_{-2}(D_4)$. They arise from rank one SCFT under the $4d/2d$ duality (cf. \cite{beem2015infinite}, \cite{shan2023mirror}). The case for $L_{-2}(D_4)$ was already considered by \cite{kawasetsu2022relaxed}, from a different point of view (Zhu's induction).
\end{example}

\begin{theorem}
    Consider the diagram
    $$
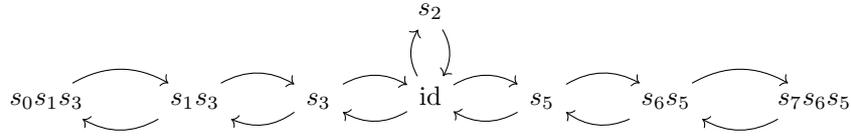
\begin{tikzcd}
&                                                  &                                               & s_2 \arrow[d, bend left]                                           &                                               &                                                  &                                \\
s_0s_1s_3 \arrow[r, bend left] & s_1s_3 \arrow[r, bend left] \arrow[l, bend left] & s_3 \arrow[r, bend left] \arrow[l, bend left] & \id \arrow[r, bend left] \arrow[u, bend left] \arrow[l, bend left] & s_5 \arrow[r, bend left] \arrow[l, bend left] & s_6s_5 \arrow[r, bend left] \arrow[l, bend left] & s_7s_6s_5 \arrow[l, bend left]
\end{tikzcd}$$
Then we have
    $$\dim\Ext^1_{\widehat{E_7}\wtMod}(L(w\cdot(-\Lambda_4)),L(w'\cdot(-\Lambda_4)))=\#\text{ of arrows from }w\text{ to }w'\text{ in the above diagram}$$
    for $w,w'\in\{\id,s_2,s_3,s_1s_3,s_0s_1s_3,s_5,s_6s_5,s_7s_6s_5\}$.

    Furthermore, all nontrivial extensions are in fact extensions of $L_{-4}(E_7)$-modules.
\end{theorem}

\begin{theorem}
    Consider the diagram
    $$
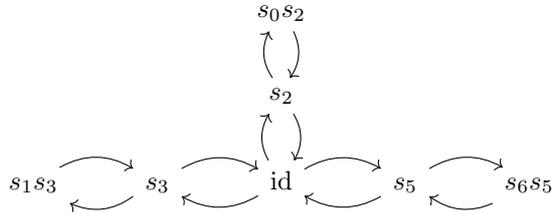
\begin{tikzcd}
&                                               & s_0s_2 \arrow[d, bend left]                                        &                                       &                             \\
                            &                                               & s_2 \arrow[d, bend left] \arrow[u, bend left]                      &                                               &                             \\
s_1s_3 \arrow[r, bend left] & s_3 \arrow[r, bend left] \arrow[l, bend left] & \id \arrow[r, bend left] \arrow[u, bend left] \arrow[l, bend left] & s_5 \arrow[r, bend left] \arrow[l, bend left] & s_6s_5 \arrow[l, bend left]
\end{tikzcd}$$
Then we have
    $$\dim\Ext^1_{\widehat{E_6}\wtMod}(L(w\cdot(-\Lambda_4)),L(w'\cdot(-\Lambda_4)))=\#\text{ of arrows from }w\text{ to }w'\text{ in the above diagram}$$
    for $w,w'\in\{\id,s_2,s_0s_2,s_3,s_1s_3,s_5,s_6s_5\}$.

    Furthermore, all nontrivial extensions are in fact extensions of $L_{-3}(E_6)$-modules.
\end{theorem}

\begin{theorem}
    Consider the diagram
    $$
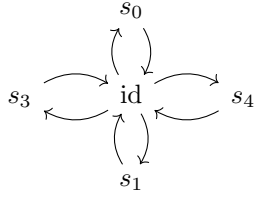
\begin{tikzcd}
& s_0 \arrow[d, bend left]                                                                &                          \\
s_3 \arrow[r, bend left] & \id \arrow[r, bend left] \arrow[u, bend left] \arrow[l, bend left] \arrow[d, bend left] & s_4 \arrow[l, bend left] \\
                         & s_1 \arrow[u, bend left]                                                                &                         
\end{tikzcd}$$
Then we have
$$\dim\Ext^1_{\widehat{D_4}\wtMod}(L(w\cdot(-\Lambda_2)),L(w'\cdot(-\Lambda_2)))=\#\text{ of arrows from }w\text{ to }w'\text{ in the above diagram}$$
    for $w,w'\in\{\id,s_0,s_1,s_3,s_4\}$.

    Furthermore, all nontrivial extensions are in fact extensions of $L_{-2}(D_4)$-modules.
\end{theorem}

\bibliographystyle{alpha}
\bibliography{bib}

\newcommand{\etalchar}[1]{$^{#1}$}
\begin{thebibliography}{BLL{\etalchar{+}}15}

\bibitem[ACK23]{arakawa2023weight}
Tomoyuki Arakawa, Thomas Creutzig, and Kazuya Kawasetsu.
\newblock {Weight representations of affine Kac--Moody algebras and small
  quantum groups}.
\newblock {\em arXiv preprint arXiv:2311.10233}, 2023.

\bibitem[AM18]{arakawa2018joseph}
Tomoyuki Arakawa and Anne Moreau.
\newblock Joseph ideals and lisse minimal {$W$}-algebras.
\newblock {\em J. Inst. Math. Jussieu}, 17(2):397--417, 2018.

\bibitem[BB81]{beilinson1981localisation}
Alexandre Beilinson and Joseph Bernstein.
\newblock Localisation de {$\mathfrak{g}$}-modules.
\newblock {\em C. R. Acad. Sci. Paris S\'er. I Math.}, 292(1):15--18, 1981.

\bibitem[BB93]{beilinson1993proof}
A.~Beilinson and J.~Bernstein.
\newblock A proof of {J}antzen conjectures.
\newblock In {\em I. {M}. {G}el'fand {S}eminar}, volume 16, Part 1 of {\em Adv.
  Soviet Math.}, pages 1--50. Amer. Math. Soc., Providence, RI, 1993.

\bibitem[BBD82]{beilinson1982faisceaux}
A.~A. Beilinson, J.~Bernstein, and P.~Deligne.
\newblock Faisceaux pervers.
\newblock In {\em Analysis and topology on singular spaces, {I} ({L}uminy,
  1981)}, volume 100 of {\em Ast\'erisque}, pages 5--171. Soc. Math. France,
  Paris, 1982.

\bibitem[BLL{\etalchar{+}}15]{beem2015infinite}
Christopher Beem, Madalena Lemos, Pedro Liendo, Wolfger Peelaers, Leonardo
  Rastelli, and Balt~C. van Rees.
\newblock Infinite chiral symmetry in four dimensions.
\newblock {\em Comm. Math. Phys.}, 336(3):1359--1433, 2015.

\bibitem[Bou02]{bourbakigroupes}
Nicolas Bourbaki.
\newblock {\em Lie groups and {L}ie algebras. {C}hapters 4--6}.
\newblock Elements of Mathematics (Berlin). Springer-Verlag, Berlin, 2002.
\newblock Translated from the 1968 French original by Andrew Pressley.

\bibitem[CD21]{campbell2021affine}
Justin Campbell and Gurbir Dhillon.
\newblock {Affine Harish-Chandra bimodules and Steinberg--Whittaker
  localization}.
\newblock {\em arXiv preprint arXiv:2108.02806}, 2021.

\bibitem[Del80]{deligne1980conjecture}
Pierre Deligne.
\newblock La conjecture de {W}eil. {II}.
\newblock {\em Inst. Hautes \'Etudes Sci. Publ. Math.}, (52):137--252, 1980.

\bibitem[Duf75]{duflo1973construction}
Michel Duflo.
\newblock Construction of primitive ideals in an enveloping algebra.
\newblock In {\em Lie groups and their representations ({P}roc. {S}ummer
  {S}chool, {B}olyai {J}\'anos {M}ath. {S}oc., {B}udapest, 1971)}, pages
  77--93. Halsted Press, New York-Toronto, Ont., 1975.

\bibitem[Hum08]{humphreys2008representations}
James~E. Humphreys.
\newblock {\em Representations of semisimple {L}ie algebras in the {BGG}
  category {$\mathcal{O}$}}, volume~94 of {\em Graduate Studies in
  Mathematics}.
\newblock American Mathematical Society, Providence, RI, 2008.

\bibitem[Kac90]{kac1990infinite}
Victor~G. Kac.
\newblock {\em Infinite-dimensional {L}ie algebras}.
\newblock Cambridge University Press, Cambridge, third edition, 1990.

\bibitem[KR22]{kawasetsu2022relaxed}
Kazuya Kawasetsu and David Ridout.
\newblock Relaxed highest-weight modules {II}: {C}lassifications for affine
  vertex algebras.
\newblock {\em Commun. Contemp. Math.}, 24(5):Paper No. 2150037, 43, 2022.

\bibitem[SXY23]{shan2023mirror}
Peng Shan, Dan Xie, and Wenbin Yan.
\newblock {Mirror symmetry for circle compactified $4d$ $\mathcal{N}=2$ SCFTs}.
\newblock {\em arXiv preprint arXiv:2306.15214}, 2023.

\bibitem[Vog80]{vogan1980ordering}
David~A. Vogan, Jr.
\newblock Ordering of the primitive spectrum of a semisimple {L}ie algebra.
\newblock {\em Math. Ann.}, 248(3):195--203, 1980.

\end{thebibliography}

\end{document}